\documentclass[12pt]{article}

\usepackage{amsthm}
\usepackage{amssymb}
\usepackage{amsmath}
\usepackage{comment}
\usepackage{thm-restate}
\usepackage{url} 
\usepackage{hyperref}
\usepackage[noabbrev,capitalise]{cleveref}

\usepackage{color}
\usepackage[normalem]{ulem}

\usepackage[margin=1in]{geometry}

\newcommand{\mc}[1]{\mathcal{#1}}

\renewcommand{\v}{\textup{\textsf{v}}}
\newcommand{\e}{\textup{\textsf{e}}}
\renewcommand{\d}{\textup{\textsf{d}}}

\theoremstyle{plain}
\newtheorem{thm}{Theorem}[section]
\newtheorem{lem}[thm]{Lemma}
\newtheorem{claim}{Claim}[thm]

\newtheorem{cor}[thm]{Corollary}

\newtheorem{conj}[thm]{Conjecture}

{\noindent \emph{Proof.} {}{#1}{}}{\hfill
	$\Diamond$\vspace{1em}}

\theoremstyle{plain} 
\newcommand{\thistheoremname}{}
\newtheorem{genericthm}[section]{\thistheoremname}

\theoremstyle{definition}
\newtheorem{definition}[thm]{Definition}

\title{Further Progress towards the List and Odd Versions of Hadwiger's Conjecture}

\author{ Luke Postle\thanks{Department of Combinatorics and Optimization, University of Waterloo, Waterloo, Ontario, Canada. Email: {\tt lpostle@uwaterloo.ca}. Canada Research Chair in Graph Theory. Partially supported by NSERC under Discovery Grant No. 2019-04304, the Ontario Early Researcher Awards program and the Canada Research Chairs program.}}

\begin{document}

\maketitle

\begin{center}
	\emph{Dedicated to the memory of Robin Thomas}
\end{center}

\begin{abstract} 
	In 1943, Hadwiger conjectured that every graph with no $K_t$ minor is $(t-1)$-colorable for every $t\ge 1$.
	In the 1980s, Kostochka and Thomason independently proved that every graph with no $K_t$ minor has average degree $O(t\sqrt{\log t})$ and hence is $O(t\sqrt{\log t})$-colorable.  Recently, Norin, Song and the author showed that every graph with no $K_t$ minor is $O(t(\log t)^{\beta})$-colorable for every $\beta > 1/4$, making the first improvement on the order of magnitude of the $O(t\sqrt{\log t})$ bound. Building on that work, we previously showed that every graph with no $K_t$ minor is $O(t (\log t)^{\beta})$-colorable for every $\beta > 0$. More specifically, they are $O(t \cdot (\log \log t)^{6})$-colorable. In this paper, we extend this work to the list and odd generalizations of Hadwiger's conjecture as follows.
	
As for the list coloring generalization, Voigt in 1993 showed that there exist planar graphs which are not $4$-list colorable. In 2011, Bar\'{a}t, Joret and Wood constructed graphs with no $K_{3t+2}$ minor which are not $4t$-list colorable for every $t \geq 1$. Thus List Hadwiger's is false but its linear relaxation remains open. The works of Kostochka and Thomason show that graphs with no $K_t$ minor are $O(t\sqrt{\log t})$-list-colorable. Recently, Norin and the author proved that every graph with no $K_t$ minor is $O(t(\log t)^{\beta})$-list-colorable for every $\beta > 1/4$. Here we prove that every graph with no $K_t$ minor is $O(t (\log t)^{\beta})$-list-colorable for every $\beta > 0$. More specifically, they are $O(t \cdot (\log \log t)^{6})$-list-colorable. 

As for the odd minor generalization, by the the early 1990s, Gerards and Seymour conjectured that every graph with no odd $K_t$ minor is $(t-1)$-colorable. In 2009, Geelen, Gerards, Reed, Seymour and Vetta showed that every graph with no odd $K_t$ minor is $O(t\sqrt{\log t})$-colorable. Recently, Norin and Song proved that every graph with no odd $K_t$ minor is $O(t(\log t)^{\beta})$-colorable for every $\beta > 1/4$. Here we prove that every graph with no odd $K_t$ minor is $O(t (\log t)^{\beta})$-colorable for every $\beta > 0$. More specifically, they are $O(t \cdot (\log \log t)^{6})$-colorable. 

\end{abstract}

\section{Introduction}

All graphs in this paper are finite and simple. Given graphs $H$ and $G$, we say that $G$ has \emph{an $H$ minor} if a graph isomorphic to $H$ can be obtained from a subgraph of $G$ by contracting edges. We denote the complete graph on $t$ vertices by $K_t$.

In 1943 Hadwiger made the following famous conjecture.

\begin{conj}[Hadwiger's conjecture~\cite{Had43}]\label{Hadwiger} For every integer $t \geq 1$, every graph with no $K_{t}$ minor is $(t-1)$-colorable. 
\end{conj}

Hadwiger's conjecture is widely considered among the most important problems in graph theory and has motivated numerous developments in graph coloring and graph minor theory. For an overview of major progress on Hadwiger's conjecture, we refer the reader to~\cite{NPS19}, and to the recent survey by Seymour~\cite{Sey16Survey} for further background.

The following is a natural weakening of Hadwiger's conjecture, which has been considered by several researchers.

\begin{conj}[Linear Hadwiger's conjecture~\cite{ReeSey98,Kaw07, KawMoh06}]\label{c:LinHadwiger} There exists a constant $C>0$ such that for every integer $t \geq 1$, every graph with no $K_{t}$ minor is $Ct$-colorable. 
\end{conj}

For many decades, the best general bound on the number of colors needed to properly color every  graph graph with no $K_t$ minor had been $O(t\sqrt{\log{t}})$, a result obtained independently by Kostochka~\cite{Kostochka82,Kostochka84} and Thomason~\cite{Thomason84} in the 1980s. The results of \cite{Kostochka82,Kostochka84,Thomason84} bound the ``degeneracy" of graphs with no $K_t$ minor. Recall that a graph $G$ is \emph{$d$-degenerate} if every non-empty subgraph of $G$ contains a vertex of degree at most $d$. A standard inductive argument shows that every $d$-degenerate graph is $(d+1)$-colorable. Thus the following bound on the degeneracy of graphs with no $K_t$ minor gives a corresponding bound on their chromatic number and even their list chromatic number. 

\begin{thm}[\cite{Kostochka82,Kostochka84,Thomason84}]\label{t:KT} Every graph with no $K_t$ minor is $O(t\sqrt{\log{t}})$-degenerate.
\end{thm}

Kostochka~\cite{Kostochka82,Kostochka84} and de la Vega~\cite{Vega83} have shown that there exist graphs with no $K_t$ minor and minimum degree $\Omega(t\sqrt{\log{t}})$. Thus the bound in \cref{t:KT} is tight. Until very recently $O(t\sqrt{\log{t}})$ remained the best general bound for the chromatic number of graphs with no $K_t$ minor when Norin, Song and the author~\cite{NPS19} improved this with the following theorem.

\begin{thm}[\cite{NPS19}]\label{t:ordinaryHadwiger}
For every $\beta > \frac 1 4$, every graph with no $K_t$ minor is $O(t (\log t)^{\beta})$-colorable.
\end{thm}

The author improved this in~\cite{Pos20,Pos20Density} as follows.

\begin{thm}\label{t:ordinaryHadwiger2}
For every $\beta > 0$, every graph with no $K_t$ minor is $O(t (\log t)^{\beta})$-colorable. More specifically, every graph with no $K_t$ minor is $O(t \cdot (\log \log t)^{6})$-colorable.
\end{thm}

\subsection{List Hadwiger's}

Let $L=\{L(v)\}_{v \in V(G)}$ be an assignment of lists of colors to  vertices of a graph $G$. We let $|L| := \min_{v\in V(G)} |L(v)|$. We say that $G$ is \emph{$L$-list colorable} if there is a choice of colors $\{c(v)\}_{v \in V(G)}$ such that $c(v) \in L(v)$, and $c(v) \neq c(u)$ for every $uv \in E(G)$. A graph $G$ is said to be \emph{$k$-list colorable} if $G$ is $L$-list colorable for every list assignment $L$ such that $|L| \geq k$. Clearly every $k$-list colorable graph is $k$-colorable, but the converse implication does not hold. 

Voigt~\cite{Voigt93} has shown that there exist planar graphs which are not $4$-list colorable. Generalizing that result, Bar\'{a}t, Joret and Wood~\cite{BJW11} constructed graphs with no $K_{3t+2}$ minor which are not $4t$-list colorable for every $t \geq 1$. These results leave open the possibility that Linear Hadwiger's Conjecture holds for list coloring, as conjectured by Kawarabayashi and Mohar~\cite{KawMoh07}.

\begin{conj}[\cite{KawMoh07}]\label{c:ListHadwiger} There exists $C>0$ such that for every integer $t \geq 1$, every graph with no $K_{t}$ minor is $Ct$-list colorable. 
\end{conj}

\cref{t:KT} implies that every graph with no $K_t$ minor is $O(t\sqrt{\log{t}})$-list colorable, which until recently was the best known upper bound for general $t$. In~\cite{NorPos20}, Norin and the author extended Theorem~\ref{t:ordinaryHadwiger} to list coloring.

\begin{thm}[\cite{NPS19}]\label{t:listHadwiger}
For every $\beta > \frac 1 4$, every graph with no $K_t$ minor is $O(t (\log t)^{\beta})$-list-colorable.
\end{thm}

We improve this further by extending Theorem~\ref{t:ordinaryHadwiger2} to list coloring as follows.

\begin{thm}\label{t:listHadwiger2}
For every $\beta > 0$, every graph with no $K_t$ minor is $O(t (\log t)^{\beta})$-list-colorable. More specifically, every graph with no $K_t$ minor is $O(t \cdot (\log \log t)^{6})$-list-colorable.
\end{thm}

\subsection{Odd Hadwiger's}

Given graphs $G$ and $H$ we say that $G$ has \emph{an odd $H$ minor} if a graph isomorphic to $H$ can be obtained from a subgraph $G'$ of $G$ by contracting a set of edges forming a cut in $G'$. Gerards and Seymour (see \cite[p. 115]{JenToft95}) conjectured the following strengthening of Hadwiger's conjecture.

\begin{conj}[Odd Hadwiger's Conjecture]\label{odd} 
For every integer  $t \geq 1$, every graph	with no odd  $K_{t}$ minor  is  $(t-1)$-colorable. 
\end{conj}

Geelen, Gerards, Reed, Seymour and Vetta~\cite{GGRSV08}  used \cref{t:KT} to show that every graph with no odd $K_t$ minor is $O(t\sqrt{\log{t}})$-colorable. In~\cite{NorSong19Odd}, Norin and Song extended Theorem~\ref{t:ordinaryHadwiger} to odd minors as follows.

\begin{thm}[\cite{NorSong19Odd}]\label{t:oddHadwiger} For every $\beta > \frac 1 4 $,
	every graph with no odd $K_t$ minor is $O(t (\log  t)^{\beta})$-colorable.
\end{thm}

We improve this further by extending Theorem~\ref{t:ordinaryHadwiger2} to odd minors as follows.

\begin{thm}\label{t:oddHadwiger2}
For every $\beta > 0$, every graph with no odd $K_t$ minor is $O(t (\log t)^{\beta})$-colorable. More specifically, every graph with no odd $K_t$ minor is $O(t \cdot (\log \log t)^{6})$-colorable.
\end{thm}

Note that bipartite graphs have no odd $K_3$ minor and unbounded list chromatic number and hence the common generalization of Theorems~\ref{t:listHadwiger2} and~\ref{t:oddHadwiger2} does not hold. However, certain aspects of their proofs will overlap. More specifically and surprisingly, we use a number of the ideas from our list proof in our odd proof, albeit with a number of additional ideas to help build an odd minor.

\subsubsection*{Notation}

We use largely standard graph-theoretical notation. We denote by $\v(G)$ and $\e(G)$ the number of  vertices and edges of a graph $G$, respectively, and denote by $\d(G)=\e(G)/\v(G)$ the \emph{density} of a non-null graph $G$. We use $\chi_{\ell}(G)$ to denote the list chromatic number of $G$, and $\kappa(G)$ to denote the (vertex) connectivity of $G$.  We denote by $G[X]$ the subgraph of $G$ induced by a set $X \subseteq V(G)$. For disjoint subsets $A,B\subseteq V(G)$, we let $G(A,B)$ denote the bipartite subgraph induced by $G$ on the parts $(A,B)$. For $F \subseteq E(G)$ we denote by $G/F$ the minor of $G$ obtained by contracting the edges of $F$.

For a positive integer $n$, let $[n]$ denote the set $\{1,2,\ldots,n\}$. The logarithms in the paper are natural unless specified otherwise.

We say that vertex-disjoint subgraphs $H$ and $H'$ of a graph $G$ are \emph{adjacent} if there exists an edge of $G$ with one end in $V(H)$ and the other in $V(H')$, and $H$ and $H'$ are \emph{non-adjacent}, otherwise.

A collection $\mc{X} = \{X_1,X_2,\ldots,X_h\}$ of pairwise disjoint subsets of $V(G)$ is a \emph{model of a graph $H$ in a graph $G$} if $G[X_i]$ is connected for every $i \in [h]$, and there exists a bijection $\phi: V(H) \to [h]$, such that $G[X_{\phi(u)}]$ and $G[X_{\phi(v)}]$ are adjacent for every $uv \in E(H)$. We let $V(\mc{X})$ denote $\bigcup_{i\in[h]} X_i$. It is well-known and not hard to see that $G$ has an $H$ minor if and only if there exists a model of $H$ in $G$. We say that a model $\mc{X}$ as above is \emph{rooted at $S$} for $S \subseteq V(G)$ if $|S|=h$  and  $|X_i \cap S|=1$ for every $i \in [h]$.

Let $H$ and $G$ be graphs. An \emph{$H$-expansion in $G$} is a function $\eta$ with
domain $V(H) \cup E(H)$ such that
\begin{itemize}
\item for every vertex $v \in V(H)$, $\eta(v)$ is a subgraph of $G$ which is a tree, and
the trees $\{\eta(v):v\in V(H)\}$ are pairwise vertex-disjoint, and
\item for every edge $e = uv \in E(H)$, $\eta(e)$ is an edge of $G$ with one end in
$V(\eta(u))$ and the other end in $V(\eta(v))$.
\end{itemize}
We call the trees $\{\eta(v):v\in V(H)\}$ the \emph{nodes} of the expansion, and denote by $\cup \eta$ the subgraph of $G$ with vertex set $\bigcup_{v\in V(H) } V(\eta(v))$ and edge set $\{E(\eta(v)) : v \in V(H)\} \cup \{\eta(e) : e \in E(H)\}$. An $H$-expansion $\eta$ is \emph{bipartite}
if $\cup \eta$ is bipartite. Moreover, we say that an $H$-expansion $\eta$ in $G$ is $S$-rooted
for $S \subseteq V (G)$ if $|S| = \v(H)$ and $|V(\eta(v)) \cap S| = 1$ for every $v \in V(H)$.
It is well-known and easy to see that $G$ has an $H$ minor if and only if
there is an $H$-expansion in $G$. We say that $G$ has a \emph{bipartite $H$ minor} if
there exists a bipartite $H$-expansion in $G$.

If $L$ is a list assignment of a graph $G$ and $L'$ is a list assignment of a subgraph $H$ of $G$, we write $L'\subseteq L$ if $L'(v)\subseteq L(v)$ for all $v\in V(H)$.

\section{Outline of Proof}\label{s:outline}

\subsection{Main Technical Theorems}

In order to state our main technical theorems, we will need the following theorem by the author found in~\cite{Pos20Density}.

\begin{thm}\label{t:newforced} There exists a constant $C=C_{\ref{t:newforced}} > 0$ such that the following holds. Let $G$ be a graph with $\d(G) \ge C$, and let $D > 0$ be a constant. Let $s=D/\d(G)$ and let $g_{\ref{t:newforced}}(s) := C(1+ \log s)^{6}$.  Then $G$ contains at least one of the following: 
\begin{description}
		\item[(i)] a minor $J$ with $\d(J) \geq D$, or
		\item[(ii)] a subgraph $H$ with $\v(H) \leq g_{\ref{t:newforced}}(s) \cdot \frac{D^2}{\d(G)}$ and $\d(H) \geq \frac{\d(G)}{g_{\ref{t:newforced}}(s)}$.
	\end{description}  
\end{thm}

In~\cite{NorSong19}, Norin and Song had proved Theorem~\ref{t:newforced} with $g(s) = s^{\alpha}$ for any $\alpha > \frac{\log(2)}{\log(3/2)} - 1 \approx  .7095$. Using that result, they showed that every graph with no $K_t$ minor is $O(t(\log{t})^{0.354})$-colorable. Shortly thereafter, in~\cite{Pos19}, the author improved this to $g(s)=s^{o(1)}$.  That result was then combined in~\cite{NPS19} with the earlier work to yield Theorem~\ref{t:ordinaryHadwiger}. The function $g(s)$ in~\cite{NPS19} was not explicitly found. It is not hard to derive an explicit function of $g(s) = 2^{O((\log s)^{2/3}+1)}$ from Lemma 2.5 in~\cite{NPS19}. The main result of~\cite{Pos20Density} is the bound listed above.

We are now ready to state our main technical theorems as follows.

\begin{thm}\label{thm:listcombined}
Every graph with no $K_t$ minor has list chromatic number at most  
$$O\left(t \cdot\left(g_{\ref{t:newforced}}\left(7 \cdot \sqrt{\log t}\right) + (\log \log t)^3\right)\right).$$
\end{thm}

\begin{thm}\label{thm:oddcombined}
Every graph with no odd $K_t$ minor has chromatic number at most  
$$O\left(t \cdot\left(g_{\ref{t:newforced}}\left(7 \cdot \sqrt{\log t}\right) + (\log \log t)^3\right)\right).$$
\end{thm}

Note that Theorem~\ref{t:listHadwiger2} follows immediately from Theorems~\ref{t:newforced} and~\ref{thm:listcombined} while Theorem~\ref{t:oddHadwiger2} follows immediately from Theorems~\ref{t:newforced} and~\ref{thm:oddcombined}.

\subsection{Outline of Proof of Main Technical Theorems}

We split the proofs of Theorem~\ref{thm:listcombined} and~\ref{thm:oddcombined} according to two main cases as determined by the following key definition.

\begin{definition}
Let $s$ be a nonnegative integer. Let $G$ be a graph and let $L$ be a list assignment of $G$. We say that $L$ is \emph{$s$-chromatic-separable} in $G$ if there exist two vertex-disjoint subgraphs $H_1,H_2$ of $G$ such that for each $i\in \{1,2\}$, there exists a list assignment $L_i \subseteq L$ of $H_i$ such that $H_i$ is not $L_i$-colorable and $|L_i(v)| \ge |L(v)|-s$ for all $v\in V(H_i)$.

If $L$ is not $s$-chromatic-separable in $G$ and $G$ is not $L$-colorable, then we say that $L$ is \emph{$s$-chromatic-inseparable} in $G$. 
\end{definition}

Theorems~\ref{thm:listcombined} and~\ref{thm:oddcombined} will follow from the following two main lemmas.

\begin{restatable}{lem}{Inseparable}\label{lem:inseparable}
There exists an integer $C=C_{\ref{lem:inseparable}} > 0$ such that the following holds: Let $t\ge 3$ be an integer. If $G$ is a graph and $L$ is a $Ct \cdot(g_{\ref{t:newforced}}(7 \cdot \sqrt{\log t}) + (\log \log t)^2)$-list-assignment of $G$ that is $Ct(\log \log t)^2$-chromatic-inseparable in $G$ and $G$ is not $L$-colorable, then $G$ contains a bipartite $K_t$ minor or an odd $K_t$ minor.
\end{restatable}

\begin{restatable}{lem}{Separable}\label{lem:separable}
There exists an integer $C=C_{\ref{lem:separable}} > 0$ such that the following holds: Let $t\ge 3$ be an integer and let $m$ be a constant such that $m\ge C t$. Let $G$ be a graph and $L$ a list-assignment of $G$ such that $|L|\ge C m \log \log t$. If for every subgraph $H$ of $G$ and every list assignment $L'\subseteq L$ of $H$ with $|L'| \ge \frac{7}{8} |L|$, we have that $L$ is $m$-chromatic-separable in $H$, then $G$ contains a $K_t$ minor.

Furthermore if $L(v) = [|L|]$ for all $v\in V(G)$, then $G$ contains a bipartite $K_t$ minor.
\end{restatable}

We are now ready to prove our first main technical theorem, Theorem~\ref{thm:listcombined} assuming Lemmas~\ref{lem:inseparable} and~\ref{lem:separable}.

\begin{proof}[Proof of Theorem~\ref{thm:listcombined}]
We prove the contrapositive. Let $m := \max\{C_{\ref{lem:inseparable}},C_{\ref{lem:separable}}\} \cdot t (\log \log t)^2$. Let $k_1 := C_{\ref{lem:inseparable}}t \cdot (g_{\ref{t:newforced}}(7 \cdot \sqrt{\log t}) + (\log \log t)^2)$, $k_2 := C_{\ref{lem:separable}} m \log \log t$and $k=\left\lceil \max \left\{\frac{8}{7} k_1, k_2\right\} \right\rceil$. Let $G$ be a graph with $\chi_{\ell}(G) \ge k$. Thus there exists a $k$-list-assignment $L$ of $G$ such that $G$ is not $L$-colorable. 

By Lemma~\ref{lem:separable}, we have that $G$ contains a $K_t$ minor as desired or that $G$ contains a subgraph $H$ and a list-assignment $L'\subseteq L$ with $|L'|\ge \frac{7}{8}|L|$ that is $m$-chromatic-inseparable in $H$. We may assume the latter case or we are done. Note that $|L'|\ge k_1$. But then by Lemma~\ref{lem:inseparable} as $|L'|\ge k_1$, we have that $H$ contains a $K_t$ minor and hence that $G$ contains a $K_t$ minor as desired. 
\end{proof}

For the odd minor version, we also need the following theorem of Geelen et al.~\cite{GGRSV08} that converts bipartite minors to odd minors if the graph is not too close to bipartite.

\begin{thm}[\cite{GGRSV08}]\label{t:oddbipartite}
If a graph $H$ has a bipartite $K_{12t}$ minor, then
either $H$ contains an odd $K_t$ minor, or there exists $X \subseteq V(H)$ with $|X| \le 8t - 2$ such that some component of $H \setminus X$ is bipartite and contains at least $8t + 2$ vertices.
\end{thm}

We are now ready to prove our second main technical theorem, Theorem~\ref{thm:oddcombined} assuming Lemmas~\ref{lem:inseparable} and~\ref{lem:separable}.

\begin{proof}[Proof of Theorem~\ref{thm:oddcombined}]
We prove the contrapositive. Let $m := \max\{C_{\ref{lem:inseparable}},C_{\ref{lem:separable}}\} \cdot t (\log \log t)^2$. Let $k_1 := C_{\ref{lem:inseparable}}t \cdot (g_{\ref{t:newforced}}(7 \cdot \sqrt{\log (12t)}) + (\log \log(12t))^2)$, $k_2 := C_{\ref{lem:separable}} m \log \log (12t)$ and $k = \left\lceil \max\left\{\frac{8}{7}k_1,k_2,8t+1\right\} \right\rceil$. Let $G$ be a graph with $\chi(G) \ge k$. 

We assume without loss of generality that $G$ is $k$-critical (that is $\chi(G)=k$ and $\chi (H) < k$ for every proper subgraph $H$ of $G$). Hence there does not exist $X \subseteq V(G)$ with $|X| \le 8t - 2$ such that some component of $G \setminus X$ is bipartite and contains at least $8t + 2$ vertices. Thus by Theorem~\ref{t:oddbipartite}, it suffices to prove that $G$ contains a bipartite $K_{12t}$ minor or an odd $K_t$ minor.

Let $L$ be the list assignment with $L(v) = [k]$ for all $v\in V(G)$. By Lemma~\ref{lem:separable} as $|L|\ge k_2$, we have that $G$ contains a bipartite $K_{12t}$ minor or that $G$ contains a subgraph $H$ and a list-assignment $L'\subseteq L$ with $|L'|\ge \frac{7}{8} k$ that is $m$-chromatic-inseparable in $H$. We may assume the last case or we are done. Note that $|L'|\ge k_1$. But then by Lemma~\ref{lem:inseparable} as $|L'|\ge k_1$, we have that $H$ contains a bipartite $K_{12t}$ minor or an odd $K_{12t}$ minor and hence that $G$ contains a bipartite $K_{12t}$ minor or an odd $K_t$ minor as desired. 
\end{proof}

\subsection{Outline of Paper}

In Section~\ref{s:prelim}, we collect all the prior results we will need in the proofs of Lemmas~\ref{lem:inseparable} and~\ref{lem:separable}. In Section~\ref{s:prelim2}, we prove a number of new auxiliary results which we will need to prove the odd and list versions. In Section~\ref{s:inseparable}, we prove Lemma~\ref{lem:inseparable} via a stronger inductive form (Lemma~\ref{lem:inseparable2}). In Section~\ref{s:separable}, we prove Lemma~\ref{lem:separable} via a stronger inductive form (Lemma~\ref{l:rooted2}). 

For the proof of Lemma~\ref{lem:inseparable}, we find a $K_{4t,4Ct\log t}$ minor for some large enough $C$ which in turn contains a bipartite $K_t$ minor. That minor is built sequentially over $O((\log t)^2)$ steps; in each step a $K_{4t, \lceil \frac{t}{\log t} \rceil}$ minor is built and then linked to the previously built minors but also via chromatic inseparability to a high list-chromatic, highly connected subgraph yet unused for building the minor. 

Crucially, we do not rebuild the minor at each step as we did for ordinary coloring in~\cite{Pos20}. Instead we ensure in each step that the path lengths are minimized and remove all vertices with too many neighbors to the paths (which will be a small set as otherwise we find a biparite $K_t$ minor from a dense unbalanced bipartite graph). This allows us to maintain a structure comprised of a small set of vertices ($O(t \cdot \text{polylog} (t))$) and a sparse subgraph ($O(t)$-degenerate) that contains the minor and such that no vertex in the complement has many neighbors into the sparse graph. This structure yields high list chromatic number in its complement at each stage. Also crucially, we invoke this structure and list coloring arguments even for the odd version so as to ensure that we can perform $O((\log t)^2)$ steps, which is needed to guarantee a bipartite $K_t$ minor.


For the proof of Lemma~\ref{lem:separable}, the $K_t$ minor is built recursively, namely by trinary recursion with a recursion depth of $O(\log \log t)$. In each level except the last, a $K_s$ minor is built by finding three vertex-disjoint high list-chromatic, highly connected subgraphs, recursively constructing a $K_{2s/3}$ minor in each subgraph and then linking the three $K_{2s/3}$ minors together. The existence of three such subgraphs is guaranteed only by means of chromatic separability. As for the last level, the existence of a $K_{t/\log t}$ minor follows from Theorem~\ref{t:KT}. For the odd version, all of this (linking, finding complete minors) is done with additional conditions guaranteeing the correct parity so as to find a bipartite $K_t$ minor.

\section{Previous Results}\label{s:prelim}

To prove Theorem~\ref{t:ordinaryHadwiger2}, we need a number of previous results as follows.

\subsection{Small Graphs}

First for the list version, we need a bound on the list chromatic number in terms of its Hall ratio as follows. The \emph{Hall ratio} of a graph $G$ is defined to be $\max_{H\subseteq G} \left\lceil \frac{\v(H)}{\alpha(H)} \right\rceil$. We need the following theorem from~\cite{NorPos20}.

\begin{thm}[\cite{NorPos20}]\label{t:listHall} There exists $C =C_{\ref{t:listHall}}  > 0$ satisfying the following.
	Let $\rho \geq 3$, and let $G$ be a graph with the Hall ratio at most $\rho$, and let $n = \v(G)$. If $n \geq 2\rho$, then
	$$\chi_l(G) \leq C\rho\log^2\left(\frac{n}{\rho}\right).$$ 	
\end{thm}

This followed rather naturally from a result of Alon~\cite{Alon92} that for every $m \geq 2$, $\chi_l(K_{m*r}) \leq O(r \log(m)).$ It follows from a result of Duchet and Meyniel~\cite{DucMey82} that the Hall ratio of graphs with no $K_t$ minor is at most $2t$ and hence we obtain the following.

\begin{thm}[\cite{NorPos20}]\label{t:listsmall}
There exists $C= C_{\ref{t:listsmall}} > 0$ satisfying the following. If $G$ is a graph with no $K_t$ minor for some $t \ge 2$, then
	$$\chi_l(G) \leq Ct\log^2\left(\frac{\v(G)}{2t}\right).$$ 
\end{thm}

Second for the odd version, we need a bound on the list chromatic number of very small graphs with no odd $K_t$ minor. 

\begin{thm}[\cite{NorSong19Odd}]\label{t:oddsmall}
There exists $C= C_{\ref{t:oddsmall}} > 0$ satisfying the following. If $G$ is a graph with no odd $K_t$ minor for some $t \ge 2$, then
	$$\chi_l(G) \leq Ct\log^2\left(\frac{\v(G)}{2t}\right).$$ 
\end{thm}

This follows from Theorem~\ref{t:listHall} and the bound of Kawarabayashi and Song~\cite{KawSong07} on the independence number of graphs with no odd $K_t$ minor, namely that $\alpha(G) \ge \frac{\v(G)}{2t}$. 

\subsection{Density Results}

We need an explicit form of Theorem~\ref{t:KT} as follows.

\begin{thm}[\cite{Kostochka82}]\label{t:density}
 	Let $t \geq 2$ be an integer. Then every graph $G$  with $\d(G) \geq 3.2 t \sqrt{\log t}$ has a $K_t$ minor.
\end{thm}

We also need an explicit form of the biparite minor version as follows.

\begin{thm}[\cite{GGRSV08}]\label{t:bipartitedensity}
 	Let $t \geq 2$ be an integer. Then every graph $G$  with $\d(G) \geq 7 t \sqrt{\log t}$ has a bipartite $K_t$ minor.
\end{thm}

For both versions, we need the following unbalanced bipartite density theorem from Norin and the author~\cite{NorPos20}.

\begin{thm}\label{t:logbip}
	There exists $C=C_{\ref{t:logbip}}>0$ such that for every $t \geq 3$ and every bipartite graph $G$ with bipartition $(A,B)$ and no $K_t$ minor we have
	\begin{equation} \label{e:logbip}
	\e(G) \le C t\sqrt{\log t} \sqrt{|A||B|}  + (t-2)\v(G).
	\end{equation}	 
\end{thm}	 

We also require the following classical result of Mader~\cite{Mader72} which ensures that every dense graph contains a highly-connected subgraph.

\begin{lem}[\cite{Mader72}]\label{l:connect}
Every graph $G$ contains a subgraph $G'$ such that $\kappa(G') \geq \d(G)/2$.
\end{lem}

Note that since every graph $G$ contains a bipartite subgraph $G'$ with $\d(G')\ge \d(G)/2$, we get that every graph $G$ contains a bipartite subgraph $G'$ with $\kappa(G')\ge \d(G)/4$.

We need the following theorem, a consequence of Theorem~\ref{t:SmallConn}, from the author's earlier work~\cite{Pos20,Pos20Density}, except that here it is stated for biparite $K_t$ minors and thus use the constant from Theorem~\ref{t:bipartitedensity} instead of Theorem~\ref{t:density}.

\begin{thm}\label{t:SmallConn}
There exists $C_{\ref{t:SmallConn}} > 0$ such that the following holds: Let $t\ge 3$ be an integer and define $f_{\ref{t:SmallConn}}(t) := 49\cdot g_{\ref{t:newforced}}(7 \sqrt{\log t}) \le C_{\ref{t:SmallConn}}\log \log t$. For every integer $k\ge t$, if $G$ is a graph with $\d(G) \ge k \cdot f_{\ref{t:SmallConn}}(t)$ and $G$ contains no bipartite $K_t$ minor, then $G$ contains a $k$-connected subgraph $H$ with $\v(H) \le t \cdot f_{\ref{t:SmallConn}}(t) \cdot \log t$.
\end{thm}

\subsection{Connectivity Results}

In addition, we need results on connectivity as developed by Bollob\'{a}s and Thomason in~\cite{BolTho96} and further developed in~\cite{NPS19, Pos20}. First we recall some definitions.

\begin{definition}
Let $\ell$ be a positive integer. Given a collection  $\mc{S} = (\{s_i,t_i\})_{i \in [\ell]}$ of pairs of vertices of a graph $G$ (where $s_i$ and $t_i$ are possibly the same) an \emph{$\mc{S}$-linkage $\mc{P}$} is a collection of paths $\{P_1,\ldots,P_{\ell}\}$ in $G$ such that for all $i\in[\ell]$,
\begin{itemize}
\item $s_i,t_i$ are the ends of $P_i$, and 
\item $P_i-\{s_i,t_i\}$ is disjoint from $\bigcup_{j\in [t]\setminus \{i\}} V(P_j)$.
\end{itemize}
We let $V(\mc{P})$ denote the set $\bigcup_{i=1}^{\ell} V(P_i)$.
\end{definition}

\begin{definition}
A graph $G$ is said to be \emph{$k$-linked} if $1 \le k \le \v(G)$ and for any vertices $s_1,\ldots, s_k, t_1,\ldots, t_k$ of $G$, there exists an $(\{s_i,t_i\})_{i\in[x]}$-linkage $\mc{P}$.
\end{definition}


We will need the following result of Bollob\'{a}s and Thomason.

\begin{thm}[\cite{BolTho96}]\label{t:linked}
 There exists $C=C_{\ref{t:linked}}> 0$ such that the following holds: If $G$ is a graph with $\kappa(G) \geq Ck$, then $G$ is $k$-linked.
\end{thm}



We need a version of this that also allows the construction of large rooted complete minors. First a definition.

\begin{definition}
Let $a,b$ be nonnegative integers. We say a graph $G$ is \emph{$(a,b)$-woven} if the following holds:
for every set of vertices $R=\{r_1,\ldots, r_a\}$ and any vertices $s_1,\ldots, s_b, t_1,\ldots,t_b$ in $V(G)$, there exists a $K_{a}$ model $\mc{M}$ in $G$ rooted at $R$ and an $(\{s_i,t_i\})_{i \in [b]}$-linkage $\mc{P}$ in $G$ such that $V(\mc{M})\cap V(\mc{P}) = R\cap \{s_i, t_i: i\in [b]\}$. 
\end{definition}   

The next lemma follows from the work of~\cite{NPS19} (more specifically as proven just by Norin and Song in~\cite{NorSong19}) as demonstrated in~\cite{Pos20}.

\begin{thm}\label{t:woven} 
There exists $C = C_{\ref{t:woven}} > 0$ such that the following holds: Let
$b \ge a \ge 0$ be nonnegative integers. If $G$ is a graph with
$$\kappa(G) \geq C \cdot (a\sqrt{\log a} +  b),$$
then $G$ is $(a,b)$-woven.
\end{thm}

Finally, we need the following useful lemma from~\cite{Pos20} about woven subgraphs. 

\begin{lem}\label{l:woven}
Let $a,b$ be nonnegative integers. Let $G$ be a graph. Let $S=\{s_1,\ldots, s_b\}$, $T=\{t_1,\ldots, t_b\}$ be multisets of vertices in $G$. Let $\mc{P}$ be an $(\{s_i,t_i\})_{i \in [b]}$-linkage in $G$. If $H$ is a subgraph of $G$ that is $(a,b)$-woven and $R=\{r_1,\ldots, r_a\}\subseteq V(H)$, then there exists an $(\{s_i,t_i\})_{i \in [b]}$-linkage $\mc{P'}$ in $G$ and a $K_a$ model $\mc{M}$ in $G$ rooted at $R$ such that $V(\mc{P'})\cap V(\mc{M}) \subseteq R\cap V(\mc{P})$ and $V(\mc{P'})\subseteq V(H)\cup V(\mc{P})$.
\end{lem}



\subsection{Connectivity and List Assignments}

We also need the following recent yet crucial result of Gir\~{a}o and Narayanan\cite{GirNar20}. 

\begin{thm}[\cite{GirNar20}]\label{t:LargeChi}
For every positive integer $k$, if $G$ is a graph with $\chi(G)\ge 7k$, then $G$ contains a $k$-connected subgraph $H$ with $\chi(H)\ge \chi(G)-6k$.
\end{thm}

Indeed, Gir\~{a}o and Narayanan showed the same results holds for list chromatic number with $7$ and $6$ replaced with $5$ and $4$ respectively. Indeed, their work actually show the following stronger theorem about list assignments.

\begin{thm}[\cite{GirNar20}]\label{t:LargeL}
For every positive integer $k$, if $G$ is a graph and $L$ is a list assignment of $G$ with $|L|\ge 4k+1$ such that $G$ is not $L$-colorable, then $G$ contains an induced $k$-connected subgraph $H$ and list assignment $L'\subseteq L$ such that $H$ is not $L$-colorable and $|L'|\ge |L|-4k$.
\end{thm}

\subsection{A Variant of Menger's}

We need a variant of Menger's theorem from the author's earlier work~\cite{Pos20}.

\begin{lem}\label{lem:MengerVariant}
Let $G$ be a graph and let $A_1,A_2,B$ be pairwise disjoint subsets of $V(G)$. If for each $i\in\{1,2\}$, there exists a
set $\mc{P}_i$ of $A_i-B$ paths in $G-A_{3-i}$ with $|\mc{P}_i|=2|A_i|$ which are pairwise disjoint in $G-A_i$ such that every vertex of $A_i$ is in exactly two paths in $\mc{P}_i$, then there exists $|A_1|+|A_2|$ vertex-disjoint $(A_1\cup A_2) - B$ paths in $G$.
\end{lem}

\section{Preliminaries}\label{s:prelim2}

\subsection{Parity Results}

For the odd version, we need parity generalizations of linkage, $k$-linked, and $(a,b)$-woven and the results about them as follows.

\begin{definition}
Let $\mc{S} = (\{(s_i,t_i)\})_{i \in [\ell]}$ be a collection of pairs of vertices of a graph $G$ (where $s_i$ and $t_i$ are possibly the same). 
For an $S$-linkage $\mathcal{P} = \{P_1,\ldots,P_{\ell}\}$, let $I$ be the set of all $i \in [\ell]$ such that $P_i$ has an odd number of edges.
Then we say that $P$ is an \emph{$(\mc{S}, I)$-parity linkage}. 
\end{definition}

\begin{definition}
A graph $G$ is said to be \emph{$k$-parity-linked} if $1 \le k \le \v(G)$ and for any vertices $s_1,\ldots, s_k, t_1,\ldots, t_k$ of $G$ and any $I \subseteq [k] \setminus \{i: s_i=t_i\}$, there exists an $((\{s_i,t_i\})_{i\in k}, I)$-linkage $\mc{P}$.
\end{definition}

Kawarabayashi and Reed\cite{KawReed09} extended Theorem~\ref{t:linked} to parity linkages as follows. 

\begin{thm}[\cite{KawReed09}]\label{t:paritylinked}
If $G$ is a graph with $\kappa(G) \geq 50k$, then at least one of the following hold:
\begin{enumerate}
\item[(i)] $G$ is $k$-parity-linked, or
\item[(ii)] there exists $X \subseteq V(G)$ such that $|X| < 4k - 3$ and $G \setminus X$ is bipartite.
\end{enumerate}
\end{thm}

We need a version of Theorem~\ref{t:paritylinked} that eliminates the second outcome by adding a list coloring assumption.

\begin{thm}\label{t:paritylinked2}
Let $G$ be a graph with $\kappa(G) \geq 50k$. Let $\ell\ge 32k$ and let $L$ be a list assignment with $L(v)\subseteq [\ell]$ for all $v\in V(G)$ and $|L|\ge \frac{7}{8}\ell$. If $G$ is not $L$-colorable, then $G$ is $k$-parity-linked.
\end{thm}
\begin{proof}
We assume that $k\ge 1$ as otherwise there is nothing to show. Apply Theorem~\ref{t:paritylinked} to $G$. If Theorem~\ref{t:paritylinked}(i) holds, then $G$ is $k$-parity-linked as desired. So we assume that Theorem~\ref{t:paritylinked}(ii) holds. That is, there exists $X \subseteq V(G)$ such that $|X| < 4k - 3$ and $G \setminus X = (A,B)$ is bipartite. 

Since $|X| \le |L|$, we have by greedily coloring that there exists an $L$-coloring $\phi$ of $X$. For each $v\in V(G)\setminus X$, let $L'(v) = L(v)\setminus \{\phi(u): u\in N_G(v)\}$. Note that 
$$|L'|\ge |L|-|X| \ge \frac{3}{4}\ell.$$ 
For each $a\in A$, note that $L'(a) \cap \{1,\ldots, \left\lfloor \frac{\ell}{2} \right\rfloor\} \ne \emptyset$; hence let $\phi(a)\in L'(a) \cap \{1,\ldots, \left\lfloor \frac{\ell}{2} \right\rfloor\}$. Similarly for each $b\in B$, note that $L'(b) \cap \{\left\lfloor \frac{\ell}{2} \right\rfloor+1,\ldots, \ell\} \ne \emptyset$; hence let $\phi(b)\in L'(b) \cap\{\left\lfloor \frac{\ell}{2} \right\rfloor+1,\ldots, \ell\}$. Now $\phi$ is an $L$-coloring of $G$, a contradiction.
\end{proof}

We now extend the definition of woven to a parity version as follows.

\begin{definition}
Let $a,b$ be nonnegative integers. We say a graph $G$ is \emph{$(a,b)$-parity-woven} if the following holds:
for every set of vertices $R=\{r_1,\ldots, r_a\}$ and $J\subseteq [a]$ and multisets of vertices $S=\{s_1,\ldots, s_b\}, T=\{t_1,\ldots,t_b\}$ in $V(G)$ and $I \subseteq [b] \setminus \{i: s_i=t_i\}$, there exists a $K_{a}$ bipartite expansion $\mc{M}$ in $G$ rooted at $R$ with color classes $A$ and $B$ such that $A\cap R = J$ and an $((\{s_i,t_i\})_{i \in [b]}, I)$-linkage $\mc{P}$ in $G$ such that $V(\mc{M})\cap V(\mc{P}) = R\cap (S\cup T)$. 
\end{definition}   

Next we prove the analogue of Lemma~\ref{l:woven} for parity as follows.

\begin{lem}\label{l:paritywoven}
Let $a,b$ be nonnegative integers. Let $G$ be a graph. Let $S=\{s_1,\ldots, s_b\}$, $T=\{t_1,\ldots, t_b\}$ be multisets of vertices in $G$ and let $I \subseteq [b] \setminus \{i: s_i=t_i\}$. Let $\mc{P}$ be an $((\{s_i,t_i\})_{i \in [b]},I)$-linkage in $G$. If $H$ is a subgraph of $G$ that is $(a,b)$-parity-woven and $R=\{r_1,\ldots, r_a\}\subseteq V(H)$ and $J\subseteq [a]$, then there exists an $((\{s_i,t_i\})_{i \in [b]},I)$-linkage $\mc{P'}$ in $G$ and a $K_{a}$ bipartite expansion $\mc{M}$ in $G$ rooted at $R$ with color classes $A$ and $B$ such that $A\cap R = J$ in $G$ rooted at $R$ such that $V(\mc{P'})\cap V(\mc{M}) \subseteq R\cap V(\mc{P})$ and $V(\mc{P'})\subseteq V(H)\cup V(\mc{P})$.
\end{lem}
\begin{proof}
For $i\in [b]$, let $P_i$ denote the path in $\mc{P}$ containing $\{s_i,t_i\}$. Let $I_0=\{i\in[b]: V(P_i)\cap V(H)\ne \emptyset \}$. For $i\in I_0$, let $s'_i$ be the vertex in $V(P_i)\cap V(H)$ closest to $s_i$ in $P_i$, and similarly let $t'_i$ be the vertex in $V(P_i)\cap V(H)$ closest to $t_i$ in $P_i$. Let $I_0'$ be the subset of $I_0$ where the subpath of $P_i$ from $s'_i$ to $t'_i$ has an odd number of edges.

Since $H$ is $(a,b)$-parity-woven, there exists a $K_{a}$ bipartite expansion $\mc{M}$ in $H$ rooted at $R$ with color classes $A$ and $B$ such that $A\cap R = J$ and an $((\{s'_i,t'_i\})_{i \in I_0}, I_0')$-linkage $\mc{P}_1$ in $H$ such that $V(\mc{M})\cap V(\mc{P}_1) = R\cap(\{s'_i:i\in I\}\cup \{t'_i:i\in I\}) \subseteq R\cap V(\mc{P})$. 

Now for each $i\in I_0$, let $P'_i$ be the path obtained from concatenating the subpath of $P_i$ from $s_i$ to $s'_i$, the path in $\mc{P}_1$ connecting $s'_i$ to $t'_i$, and the subpath of $P_i$ from $t'_i$ to $t_i$. Note that by construction $V(\mc{P'}) \subseteq V(H)\cup V(\mc{P})$. Hence $\mc{P}=(P_i: i\in [b])$ and $\mc{M}$ are as desired.
\end{proof}

We prove the analogue of Theorem~\ref{t:woven} for parity as follows.

\begin{thm}\label{t:paritywoven} 
There exists $C = C_{\ref{t:paritywoven}} > 0$ such that the following holds: Let
$b \ge a \ge 0$ be nonnegative integers. If $G$ is a graph with
$$\kappa(G) \geq C \cdot (a\sqrt{\log a} + b),$$
then at least one of the following holds:
\begin{enumerate}
\item[(i)] $G$ is $(a,b)$-parity-woven, or
\item[(ii)] there exists $X \subseteq V(G)$ such that $|X| < 4(a+b) - 3$ and $G \setminus X$ is bipartite.
\end{enumerate}
\end{thm}
\begin{proof}
Let $C = C_{\ref{t:paritywoven}} = \max\{ 40C_{\ref{t:woven}}+2, 100\}$. If $a\le 1$, the result follows from Therem~\ref{t:paritylinked}. So we assume that $a\ge 2$.

Let $R=\{r_1,\ldots, r_a\}$ be a set of vertices of $G$, let $J\subseteq [a]$, let $s_1,\ldots, s_b, t_1,\ldots,t_b$ be any vertices in $G$ and let $I \subseteq [b] \setminus \{i: s_i=t_i\}$.

Let $G' = G\setminus (R\cup \{s_i,t_i: i\in [b]\})$. Note that $G$ has minimum degree at least $\kappa(G)$ and hence $G'$ has minimum degree at least $\kappa(G) - (a+2b)$. Thus $\d(G') \ge \frac{\kappa(G) - (a+2b)}{2}$. Now there exists a bipartite subgraph $G''$ of $G'$ with $\d(G'')\ge \d(G')/2$. By Lemma~\ref{l:connect}, there exists a bipartite subgraph $H=(A,B)$ of $G''$ with $\kappa(H)\ge \d(G'')/2 \ge \frac{\kappa(G) - (a+2b)}{8} \ge 5C_{\ref{t:woven}}\cdot (a\sqrt{\log a} + b)$.

Let $R'=\{r_1',r_2',\ldots, r_a'\}$ be a subset of $A$. Note that since $C\ge 100$, then $\kappa(G)\ge 50(a+b)$. Apply Theorem~\ref{t:paritylinked} to $G$ with $k:= a+b$. If Theorem~\ref{t:paritylinked}(ii) holds, then (ii) holds for $G$ as desired. So we assume that Theorem~\ref{t:paritylinked}(i) holds for $G$. That is, $G$ is $(a+b)$-parity-linked. 

For all $i\in[a]$, let $s_{b+i} = r_i$ and $t_{b+i}=r_i'$. Let $I' = I\cup \{b+j: j\in J\}$. As $G$ is $(a+b)$-parity-linked, there exists an $((\{s_i,t_i\})_{i \in [a+b]}, I')$-linkage $\mc{P'}=\{P_1,\ldots P_{a+b}\}$ in $G$. 

Since $\kappa(H)\ge 5C_{\ref{t:woven}}\cdot (a\sqrt{\log a} + b) \ge C_{\ref{t:woven}}\cdot (a\sqrt{\log a} + (2a+2b))$ as $a\ge 2$, it follows from Theorem~\ref{t:woven} that $H$ is $(a,2a+2b)$-woven. Let $I_0 = \{ i \in [a+b]: V(P_i)\cap V(H)\ne \emptyset \}$.

For each $i\in I_0$, let $s_i'$ denote the vertex in $V(P_i)\cap V(H)$ closest in $P_i$ to $s_i$ and let $u_i$ denote the vertex in $V(P_i)\cap V(H)$ closest in $P_i$ to $t_i$. 

We say a path $P$ in $G$ with both ends $x,y$ in $H$ is \emph{$H$-parity-preserving} if $x$ and $y$ are in the same part of $H$ and $|E(P)|$ is even, or $x$ and $y$ are in different parts of $H$ and $|E(P)|$ is odd.

For each $i\in I_0$, let $Q_i$ denote the subpath of $P_i$ from $s_i'$ to $u_i$. For each $i\in I_0$, let $I_0'$ denote the set of $i\in I$ where $Q_i$ is $H$-parity-preserving. For $i\in I_0'$, let $t_i' := u_i$. For $i\in I_0\setminus I_0'$, there exists a component $Q_i'$ of $Q_i\setminus E(H)$ that is not $H$-parity-preserving; we let $t_i'$ denote the end of $Q_i'$ closest to $s_i'$ in $Q_i$, we let $s_{a+b+i}'$ denote the end of $Q_i'$ closest to $u_i$ in $Q_i$, and we let $t_{a+b+i}' := u_i$.

Let $I_1 = I_0' \cup \{i,a+b+i: i\in I_0\setminus I_0'\}$. Let $S'$ be the multiset $\{s_i': i\in I_1\}$ and $T'$ be the multiset $\{t_i': i\in I_1\}$. As $H$ is $(a,2a+2b)$-woven, there exists a $K_{a}$ expansion $\mc{M'}$ in $H$ rooted at $R'$ and an $((\{s_i',t_i'\})_{i \in I_1})$-linkage $\mc{P''}$ in $H$ such that $V(\mc{M'})\cap V(\mc{P''}) = R'\cap (S'\cup T')$.

For each $i\in [a+b]\setminus I_0$, let $P_i':= P_i$. For each $i\in I_0'$, let $P_i'$ be the path obtained from concatenating the subpath of $P_i$ from $s_i$ to $s_i'$, the path $P''_i$ in $\mc{P}''$ from $s_i'$ to $t_i'$, and the subpath of $P_i$ from $t_i'$ to $t_i$. Similarly for each $i\in I_0\setminus I_0'$, let $P_i'$ be the path obtained from concatenating the subpath of $P_i$ from $s_i$ to $s_i'$, the path $P_i''$ in $\mc{P}''$ from $s_i'$ to $t_i'$, the path $Q_i'$ from $t_i'$ to $s_{a+b+i}'$, the path $P_{a+b+i}''$ in $\mc{P}''$ from $s_{a+b+i}'$ to $t_{a+b+i}'$, and the path from $t_{a+b+i}'$ to $t_i$.

For each $i\in a$, let $M_i = M_i' \cup P_{b+i}$. Let $\mc{M} =\{M_1,\ldots, M_a\}$. Note that $\mc{M}$ is a $K_{a}$ bipartite expansion in $G$ rooted at $R$. Moreover, let $A', B'$ be the biparitition of $V(\mc{M})$ such that $A\cap V(\mc{M}')\subseteq A'$. Then $A'\cap R = J$. In addition, $\mc{P_0}=\{P_1',\ldots, P_b'\}$ is an $((\{s_i,t_i\})_{i \in [b]}, I)$-linkage in $G$ such that $V(\mc{M})\cap V(\mc{P_0}) = R\cap (S\cup T)$ as desired. 
\end{proof}

Finally, we need a version of Theorem~\ref{t:paritywoven} that eliminates the second outcome by adding a list coloring assumption.

\begin{thm}\label{t:paritywoven2}
There exists $C = C_{\ref{t:paritywoven2}} > 0$ such that the following holds: Let
$b \ge a \ge 0$ be nonnegative integers. Let $G$ be a graph with
$$\kappa(G) \geq C \cdot (a\sqrt{\log a} + b).$$
Let $k\ge 32(a+b)$ and let $L$ be a list assignment with $L(v)\subseteq [k]$ for all $v\in V(G)$ and $|L|\ge \frac{7}{8}k$. If $G$ is not $L$-colorable, then $G$ is $(a,b)$-parity-woven.
\end{thm}
\begin{proof}
Let $C_{\ref{t:paritywoven2}} = C_{\ref{t:paritywoven}}$. We assume that $a+b\ge 1$ as otherwise there is nothing to show. Apply Theorem~\ref{t:paritywoven} to $G$. If Theorem~\ref{t:paritywoven}(i) holds, then $G$ is $(a,b)$-parity-woven as desired. So we assume that Theorem~\ref{t:paritywoven}(ii) holds. That is, there exists $X \subseteq V(G)$ such that $|X| < 4(a+b) - 3$ and $G \setminus X = (A,B)$ is bipartite. 

Since $|X| \le |L|$, we have by greedily coloring that there exists an $L$-coloring $\phi$ of $X$. For each $v\in V(G)\setminus X$, let $L'(v) = L(v)\setminus \{\phi(u): u\in N_G(v)\}$. Note that $|L'|\ge |L|-|X| \ge \frac{3}{4}k$. For each $a\in A$, note that $L'(a) \cap \{1,\ldots, \left\lfloor \frac{k}{2} \right\rfloor\} \ne \emptyset$; hence let $\phi(a)\in L'(a) \cap \{1,\ldots, \left\lfloor \frac{k}{2} \right\rfloor\}$. Similarly for each $b\in B$,  note that $L'(b) \cap \{\left\lfloor \frac{k}{2} \right\rfloor+1,\ldots, k\} \ne \emptyset$; hence let $\phi(b)\in L'(b) \cap\{\left\lfloor \frac{k}{2} \right\rfloor+1,\ldots, k\}$. Now $\phi$ is an $L$-coloring of $G$, a contradiction.
\end{proof}

\subsection{More Density Results}

Theorem~\ref{t:logbip} has the following useful corollary.

\begin{thm}\label{t:logbip2}
There exists $C=C_{\ref{t:logbip2}}>0$ such that for every $t \geq 3$ and every bipartite graph $G$ with bipartition $(A,B)$. If $|B|\ge C(\log t)|A|$ and every vertex in $B$ has at least $2t$ neighbors in $A$, then $G$ has a bipartite $K_t$ minor.
\end{thm}	 
\begin{proof}
Let $C_{\ref{t:logbip2}}= \max\{2C_{\ref{t:logbip}}^2, 2\}$. Note that $|B|\ge C_{\ref{t:logbip2}}(\log t)|A|\ge 2|A|$ and hence $|B|\ge \frac{2}{3}\v(G)$. Since every vertex in $B$ has at least $2t$ neighbors in $A$, we find that $\e(G) \ge 2t|B|$. Note that $G$ has no $K_t$ minor since $G$ has no bipartite $K_t$ minor and $G$ is bipartite. By Theorem~\ref{t:logbip} as $G$ no $K_t$ minor, we have that
$$\e(G) \le C_{\ref{t:logbip}} t\sqrt{\log t} \sqrt{|A||B|}  + (t-2)\v(G).$$
Since $|B|\ge C(\log t)|A|$, we find that
$$\e(G) \le \frac{C_{\ref{t:logbip}}}{\sqrt{C_{\ref{t:logbip2}}}} t |B|  + (t-2)\v(G) \le \frac{1}{2} t|B| + \frac{3}{2} (t-2)|B| < 2t|B|,$$
a contradiction.
\end{proof}

Theorem~\ref{t:logbip2} has the following also useful corollary.

\begin{thm}\label{t:logbip3}
There exists $C=C_{\ref{t:logbip3}}>0$ such that: If $G$ has a $K_{4t,\lceil 4Ct\log t\rceil }$ minor, then $G$ has a bipartite $K_t$ minor.
\end{thm}
\begin{proof}
Let $C_{\ref{t:logbip3}}= C_{\ref{t:logbip2}}$. Let $a= 4t$ and $b=\lceil 4Ct\log t\rceil$. Let $H=K_{a,b}=(A,B)$ where $|A|=a$ and $|B|=b$. Let $A=\{u_1,\ldots, u_a\}$ and $B=\{v_1,\ldots, v_b\}$. Let $\eta$ be an expansion of an $H$ minor of $G$. We assume without loss of generality that $\cup \eta = G$.

Let $\phi$ be a $2$-coloring of $\bigcup_{i\in [a]} \eta(u_i)$. For each $i\in [a], j\in [b]$, let $w_{i,j}$ denote the end of $\eta(u_iv_j)$ in $\eta(V(u_i))$. For each $j\in [b]$, let $D_{k,j} = \{i\in [a]: \phi(w_{i,j})=k\}$. For each $j\in [b]$, there exists $c_j\in \{1,2\}$ such that 
$$|D_{c_j,j}| \ge \frac{a}{2}.$$
Let $G'= \cup \eta - \{\eta(u_iv_j): i\in [a]\setminus D_{c_j,j}\}$. Note that $G'$ is bipartite. Moreover, $G'$ is an $H'$-expansion for some subgraph $H'$ of $H$ where every vertex in $B$ has at least $2t$ neighbors in $A$ in $H'$. Since $|B| \ge  C(\log t)|A|$, we have by Theorem~\ref{t:logbip2} that $G'$ contains a bipartite $K_t$ minor and hence so does $G$.
\end{proof}

\subsection{List Coloring Results and a Geodesic Lemma}

We need two additional auxiliary lemmas for the inseparable case.

\begin{lem}\label{l:listprob}
There exists $C=C_{\ref{l:listprob}}>0$ such the following holds: Let $G$ be a graph with neither a bipartite $K_t$ minor nor an odd $K_t$ minor. Let $X,Y \subseteq V(G)$ such that $X\cap Y=\emptyset$, $G[Y]$ is $d$-degenerate and every vertex in $G-X-Y$ has at most $d$ neighbors in $Y$. If $L$ is a list assignment of $G$ with 
$$|L| > 6Ct \log^2 \left( \frac{\max\{|X|,4t\}}{2t} \right) + d + 2t$$
 such that $G$ is not $L$-colorable, then there exists a list assignment $L'\subseteq $ of $G-X-Y$ such that $G-X-Y$ is not $L$-colorable and
$$|L'|\ge |L| - (6Ct \log( \max\{|X|,4t\}/ 2t)^2 + d+2t).$$
\end{lem}
\begin{proof}
Let $C_{\ref{l:listprob}}=\max\{C_{\ref{t:oddsmall}},120(1+C_{\ref{t:logbip2}})\}$. We assume without loss of generality that $|L(v)|=|L|$ for every vertex $v$ of $G$. Since $G$ has no odd $K_t$ minor, we have by Theorem~\ref{t:oddsmall} that $\chi_{\ell}(G[X])\le 2C_{\ref{t:oddsmall}}t \log(\max\{|X|,4t\}/ 2t)^2$. Let $r=4C_{\ref{l:listprob}}t (\log(\max\{|X|,4t\}/ 2t))^2$. Let $p = \frac{r}{|L|}$. 

Let $C_0$ be a subset of $\bigcup_{v\in V(G)} L(v)$ where every color is selected independently with probability $p$. For every $v\in V(G)$, $E[|L(v)\cap C_0|] = p|L| = r$. By the Chernoff bounds, the probability that $|L(v)\cap C_0| < \frac{r}{2}$ is at most $e^{-r/8}$ and the probability that $|L(v)\cap C_0| > \frac{3r}{2}$ is at most $e^{-r/10}$.

Let $Z  = \{z\in V(G)-X: |N_G(z)\cap X|\ge 2t\}$. As $G$ has no bipartite $K_t$ minor, we have by Theorem~\ref{t:logbip2} that $|Z|\le C_{\ref{t:logbip2}}|X|\log t$. It follows from the union bound that the probability that every $x\in X$ satisfies $|L(x)\cap C_0| \ge \frac{r}{2}$ and every vertex $z\in Z$ satisfies $|L(z)\cap C_0| \le \frac{3r}{2}$ is at most

$$1 - (|X|+|Z|)e^{-r/10} \le 1 - |X|(1+C_{\ref{t:logbip2}}\log t)e^{-r/10}.$$

If $|X|\ge 4t^2$, then $r\ge C_{\ref{l:listprob}} t (\log |X|)^2$ and hence $e^{r/10} \ge |X|^{(1+C_{\ref{t:logbip2}})t}> |X|(1+C_{\ref{t:logbip2}}\log t)$. Similarly if $|X| \le 4t^2$, then $e^{r/10} \ge e^{12t(1+C_{\ref{t:logbip2}})} > 4t^3(1+C_{\ref{t:logbip2}}) \ge |X|(1+C_{\ref{t:logbip2}}\log t)$. In either case, we find that the probability is strictly positive.

Hence there exists a choice of $C_0$ such that every $x\in X$ satisfies $|L(x)\cap C_0| \ge \frac{r}{2} \ge \chi_{\ell}(G[X])$ and every vertex $z\in Z$ satisfies $|L(z)\cap C_0| \le \frac{3r}{2}$. For $x\in X$, let $L_0(x) = L(x)\cap C_0$. Since $|L_0|\ge \chi_{\ell}(G[X])$, there exists an $L_0$-coloring $\phi$.

For $v\in V(G)-X$, let $L_1(v) = L(v)\setminus C_0$ if $v\in Z$ and let $L_1(v) = L(v)\setminus \{\phi(u): u\in N_G(v)\cap X\}$ otherwise. Since $G$ is not $L$-colorable, it follows that $G-X$ is not $L_1$-colorable. Note that $L_1(v)\subseteq L(v)$ for every $v\in V(G)-X$. Moreover, 
$$|L_1| \ge |L| - (6C_{\ref{l:listprob}}t \log(\max\{|X|,4t\}/ 2t)^2 + 2t).$$ 
Since $|L| > 6C_{\ref{l:listprob}}t \log(\max\{|X|,4t\}/ 2t)^2 + d + 2t$ and $G[Y]$ is $d$-degenerate, there exists an $L_1$-coloring $\phi_1$ of $Y$.

For $v\in V(G)-X-Y$, let $L'(v) = L_1(v)\setminus \{\phi_1(u): u\in N_G(v)\cap Y\}$. Since $G$ is not $L$-colorable, it follows $G-X-Y$ is not $L'$-colorable. Note that $L'(v)\subseteq L(v)$ for every $v\in V(G)-X-Y$. Moreover, 

$$|L'| \ge |L_1|-d \ge |L| - (6C_{\ref{l:listprob}}t \log^2(\max\{|X|,4t\}/ 2t) + d+2t),$$
and hence $L'$ is as desired.
\end{proof}

We need the following corollary of Theorem~\ref{t:SmallConn} and Lemma~\ref{l:listprob}.

\begin{cor}\label{c:SmallConn}
Let $t\ge 3$ be an integer. Let $k\ge t$ be an integer and let $r=\lceil \log t \rceil+1$. Let $G$ be a graph with neither a bipartite $K_t$ minor nor an odd $K_t$ minor. If $L$ is a list assignment of $G$ such that $|L| \ge 2k \cdot f_{\ref{t:SmallConn}}(t) + 6(C_{\ref{l:listprob}}+1)t (\log f_{\ref{t:SmallConn}}(t) + \log \log t + 1)^2$ and $G$ is not $L$-colorable, then $G$ contains $r$ vertex-disjoint $k$-connected subgraphs $H_1,\ldots, H_{r}$ with $\v(H_i) \le t \cdot f_{\ref{t:SmallConn}}(t) \cdot \log t$ for every $i\in [r]$.
\end{cor}
\begin{proof}
Suppose not. Let $H_1, \ldots, H_s$ be a maximal collection of vertex-disjoint $k$-connected subgraphs such that $\v(H_i) \le t \cdot f_{\ref{t:SmallConn}}(t) \cdot \log t$ for every $i\in [s]$. Since the lemma does not hold, $s < r$. Let $X = \bigcup_{i=1}^{s} V(H_i)$. Note that $|X| \le r t \cdot f_{\ref{t:SmallConn}}(t) \cdot \log t$. Thus
\begin{align*}
6C_{\ref{l:listprob}}t \log^2 \left( \frac{\max\{|X|,4t\}}{2t} \right) + 2t &\le 6C_{\ref{l:listprob}}t \log^2 (r \cdot f_{\ref{t:SmallConn}}(t) \cdot \log t) + 2t \\
&< 6(C_{\ref{l:listprob}}+1)t (\log f_{\ref{t:SmallConn}}(t) + \log \log t + 1)^2 \\
&\le |L|.
\end{align*}

By Lemma~\ref{l:listprob} with $Y=\emptyset$ and $d=0$, there exists a list assignment $L'\subseteq L$ of $G-X$ such that $G-X$ is not $L'$-colorable and
$$|L'|\ge |L| - (6C_{\ref{l:listprob}}t \log( \max\{|X|,4t\}/ 2t)^2 + d+2t).$$
But then $|L'| > 2k \cdot f_{\ref{t:SmallConn}}(t)$. Let $G'$ be a subgraph of $G-X$ such that $G'$ is not $L'$-colorable and subject to that $|V(G')|$ is minimized. It follows that $\delta(G') \ge |L'|$ and hence $\d(G') \ge |L'|/2 \ge k \cdot f_{\ref{t:SmallConn}}(t)$. 

By Theorem~\ref{t:SmallConn} as $G$ contains no bipartite $K_t$ minor, we find that $G'$ contains a $k$-connected subgraph $H_{s+1}$ with $\v(H_{s+1}) \le t \cdot f_{\ref{t:SmallConn}}(t) \cdot \log t$, contradicting the maximality of $H_1,\ldots, H_s$.
\end{proof}


\begin{definition}
Let $\ell$ be a positive integer. Let $G$ be a graph and let $A,B$ be disjoint vertex sets of $G$. Let $\mc{P}=\{P_1,\ldots,P_{\ell}\}$ be a set of vertex-disjoint $A-B$ paths. We say $\mc{P}$ is \emph{geodesic} if $\sum_{i\in [\ell]}$ is minimized over all such sets of paths. We say $\mc{P}$ is \emph{express} if
\begin{enumerate}
\item[(i)] for every $i\in [\ell]$ and every vertex $v$ in $V(G)\setminus V(\mc{P})$, $|N_G(v)\cap V(P_i)|\le 3$, and
\item[(ii)] $G[V(\mc{P})]$ is $(2\ell-1)$-degenerate.
\end{enumerate} 
\end{definition}

\begin{lem}\label{l:geodesic1}
Let $\ell$ be a positive integer. Let $G$ be a graph and let $A,B$ be disjoint vertex sets of $G$. If $\mc{P}=\{P_1,\ldots,P_{\ell}\}$ is a geodesic set of vertex-disjoint $A-B$ paths, then $\mc{P}$ is express.
\end{lem}
\begin{proof}
Suppose not. For $i\in [\ell]$, let $s_i$ denote the end of $P_i$ in $A$ and let $t_i$ denote the end of $P_i$ in $B$.

First suppose there exists $i\in [\ell]$ and a vertex $v$ in $V(G)\setminus V(\mc{P})$ such that $|N_G(v)\cap V(P_i)|\ge 4$. Thus there exists neighbors $v_1,v_2$ of $v$ in $P_i$ such that the distance between $v_1$ and $v_2$ in $P_i$ is at least $3$. We assume without loss of generality that $v_1$ is closer in $P_i$ to $s_i$ than $v_2$ is in $P_i$ to $s_i$. Let $Q_1$ be the subpath of $P_i$ from $s_i$ to $v_1$ and let $Q_2$ be the subpath of $P_i$ from $v_2$ to $t_i$. If $v\in A$, let $P_i' = vv_2Q_2t_i$; if $v\in B$, let $P_i'=s_iQ_1v_1v$; otherwise, let $P_i' = s_iQ_1v_1vv_2Q_2t_i$. In all cases, we note that $P_i'$ is an $A-B$ path such that $|E(P_i')| < |E(P_i)|$. Moreover, $P_i'$ is vertex-disjoint from each path in $\mc{P}\setminus \{P_i\}$. Hence $\mc{P'} = (\mc{P}\setminus \{P_i\})\cup \{P_i'\}$ is a set of $\ell$ vertex-disjoint paths contradicting the minimality of $\mc{P}$.

So we assume that $G[V(\mc{P})]$ is not $(2\ell-1)$-degenerate. That is, there exists a subgraph $H$ of $G[V(\mc{P})]$ with minimum degree at least $2\ell$. For each $i\in [\ell]$ such that $|V(H)\cap V(P_i)|\ge 2$, let $u_i$ be the vertex in $V(H)\cap V(P_i)$ closest to $s_i$ in $P_i$ and let $v_i$ be the vertex in $V(H)\cap V(P_i)$ that is second closest to $s_i$ in $P_i$. 

Consider an auxiliary digraph $D$ with $V(D) = \{u_i: |V(H)\cap V(P_i)|\ge 2\}$ and $E(D) = \{ u_iu_j: (N_G(u_i)\cap V(H)\cap V(P_j))\setminus \{u_j,v_j\}\ne \emptyset \}$. Since $|V(H)|\ge 2\ell+1$, it follows that $V(D)\ne \emptyset$. Note that by the minimality of $\mc{P}$, it follows that $P_i$ is an induced path for each $i\in [\ell]$ and hence $(N_G(u_i)\cap V(H)\cap V(P_i))\setminus \{u_i,v_i\} =\emptyset$. Since $H$ has minimum degree at least $2t$, we have that $|N_G(u_i)\cap V(H)|\ge 2t$ for every $i\in [\ell]$. Thus for every $i\in[\ell]$, there exists $j\in [\ell]\setminus\{i\}$ such that $u_iu_j \in E(D)$. Thus $D$ has minimum outdegree at least 1. 

It follows that $D$ contains a directed cycle $C=u_{i_1}u_{i_2}\ldots u_{i_k}$. We assume without loss of generality that $C=u_1\ldots u_k$. For each $i\in [k]$, let $Q_i$ be the subpath of $P_i$ from $s_i$ to $u_i$. For each $i\in [k]\setminus \{1\}$, let $w_i \in N_G(u_{i-1})\cap V(H)\cap V(P_i))\setminus \{u_i,v_i\}$ (such exists since $u_{i-1}u_i\in E(D)$) and let $Q'_{i}$ be the subpath of $P_i$ from $w_i$ to $t_i$. Similarly, let $w_1\in N_G(u_{k})\cap V(H)\cap V(P_1))\setminus \{u_1,v_1\}$ (such exists since $u_{k}u_1\in E(D)$) and let $Q'_{1}$ be the subpath of $P_1$ from $w_1$ to $t_1$.

For each $i\in [k-1]$, let $P_i'= s_iQ_iu_iw_{i+1}Q_{i+1}'t_{i+1}$. Similarly let $P_k' = s_kQ_ku_kw_1Q_1't_1$. For $i\in \{k+1,\ldots, \ell\}$, let $P_i'=P_i$. Let $\mc{P'}=\{P_1',\ldots, P_k'\}$. Note that $\mc{P'}$ is a set of $\ell$ vertex-disjoint $A-B$ paths. Moreover, $\sum_{i=1}^{\ell} |E(P_i')| < \sum_{i=1}^{\ell} |E(P_i)|$, contradicting the minimality of $\mc{P}$.
\end{proof}

\begin{definition}
Let $G$ be a graph, let $S=\{s_1,\ldots, s_{\ell}\}$ and $T=\{t_1,\ldots, t_{\ell}\}$ be multisets of vertices of $G$. Let $\mc{P}=\{P_1,\ldots,P_{\ell}\}$ be an $(\{s_i,t_i\})_{i\in [\ell]}$-linkage in $G$. We say $\mc{P}$ is \emph{geodesic} if $\sum_{i\in [\ell]} |E(P_i)|$ is minimized among all such linkages.  We say $\mc{P}$ is \emph{express} if
\begin{enumerate}
\item[(i)] for every $i\in [\ell]$ and every vertex $v$ in $V(G)\setminus V(\mc{P})$, $|N_G(v)\cap V(P_i)|\le 3$, and
\item[(ii)] $G[V(\mc{P})]\setminus (S\cup T)$ is $(2\ell-1)$-degenerate.
\end{enumerate} 
\end{definition}

\begin{lem}\label{l:geodesic2}
Let $G$ be a graph, let $S=\{s_1,\ldots, s_{\ell}\}$ and $T=\{t_1,\ldots, t_{\ell}\}$ be multisets of vertices of $G$. If $\mc{P}$ is a geodesic $(\{s_i,t_i\})_{i\in [\ell]}$-linkage in $G$, then $\mc{P}$ is express.
\end{lem}
\begin{proof}
By creating twins of vertices in $S\cup T$ as necessary, we may assume that $S$ and $T$ are sets rather than multisets and that $S$ and $T$ are disjoint.
Then the lemma follows from Lemma~\ref{l:geodesic1} by letting $A=S$ and $B=T$.
\end{proof}

\section{Finding a Minor in Chromatic-Inseparable List Assignments}\label{s:inseparable}

We prove a stronger inductive form of Lemma~\ref{lem:inseparable} (Lemma~\ref{lem:inseparable2}), that will allow us to find a $K_{4t,4C_{\ref{t:logbip3}}t\log t}$ minor and hence by Theorem~\ref{t:logbip3} a biparite $K_t$ minor. First some definitions. 

\begin{definition}
An ordered pair of collections $(\mc{A,B}) = \{A_1,A_2,\ldots,A_s,B_1,\ldots, B_t\}$ of pairwise disjoint subsets of $V(G)$ is a \emph{model of $K_{s,t}$ in a graph $G$} if $G[A_i]$ is connected for every $i \in [s]$, $G[B_j]$ is connected for every $j\in [t]$ and $A_i$ is adjacent to $B_j$ for every $i\in [s], j\in [t]$. 
\end{definition}

The following concepts are crucial to the proof of Lemma~\ref{lem:inseparable}.

\begin{definition}
Let $(\mc{A,B})$ be a $K_{s,t}$ model in a graph $G$. We say a subgraph $H$ of $G$ is \emph{$\mc{B}$-tangent} to $(\mc{A,B})$ if $V(H)\cap A_i = \emptyset$ for every $i\in [s]$ and $|V(H)\cap B_j|=1$ for every $j\in [t]$.

\end{definition}

\begin{definition}
A \emph{railroad} $\mathcal{R} = (S, \mathcal{T}_1, \mathcal{T}_2)$ in a graph $G$ of \emph{width} $t$ and \emph{length} $\ell$ is a subset $S$ (called the \emph{stations}) and two sets of paths $\mathcal{T}_1$ and $\mathcal{T}_2$ (called the \emph{inbound} and \emph{outbound tracks} respectively) where
\begin{itemize}
\item for each $i\in \{1,2\}$, the paths in $\mathcal{T}_i$ are pairwise vertex-disjoint and $|\mathcal{T}_i|\le 18t\ell$,
\item every vertex in $G-(S\cup \bigcup_{P\in \mathcal{T}_1\cup \mathcal{T}_2} V(P))$ has at most $3$ neighbors in any path in $\mathcal{T}_1\cup \mathcal{T}_2$,
\item every vertex in $G-(S\cup \bigcup_{P\in \mathcal{T}_1\cup \mathcal{T}_2} V(P))$ has a neighbor in at most $4t$ paths of $\mathcal{T}_1$ and at most $4t$ paths of $\mathcal{T}_2$, and
\item $G[\bigcup_{P\in \mathcal{T}_1\cup \mathcal{T}_2} V(P)]$ is $96t$-degenerate.
\end{itemize}
For convenience, we let $V(\mc{T}_1)$ denote $\bigcup_{P\in \mathcal{T}_1} V(P)$, $V(\mc{T}_2)$ denote $\bigcup_{P\in \mathcal{T}_2} V(P)$, and $V(\mc{R})$ denote $S\cup V(\mc{T}_1)\cup V(\mc{T}_2$.
\end{definition}

We are now ready to state and prove our stronger inductive form of Lemma~\ref{lem:inseparable}.

\begin{lem}\label{lem:inseparable2}
There exists an integer $C=C_{\ref{lem:inseparable2}} > 0$ such that the following holds: Let $t\ge 3$ be an integer and let $T\in[0, \lceil (\log t)^2 \rceil]$ be an integer. If $G$ is a graph that has neither an odd $K_t$ minor nor a bipartite $K_t$ minor and $L$ is a $Ct \cdot(f_{\ref{t:SmallConn}}(t) + (\log \log t)^2)$-list-assignment of $G$ that is $Ct(\log \log t)^2$-chromatic-inseparable in $G$ and $G$ is not $L$-colorable, then $G$ contains all of the following:
\begin{itemize}
\item a \emph{railroad} $\mathcal{R} = (S, \mathcal{T}_1, \mathcal{T}_2)$ of \emph{width} $t$ and \emph{length} $T$ where 
$$|S| \le 2(T+1)t + 3 T t \cdot f_{\ref{t:SmallConn}}(t) \cdot \log^2 t + 72 C_{\ref{t:logbip2}} (\log t) t \cdot \frac{T(T+1)}{2},$$ 
\item a $K_{s,t}$ model $(\mc{A,B})$ where $s = T \lceil \frac{t}{\log t}\rceil$ and $V(\mc{A,B})\subseteq V(\mc{R})$, and
\item a $\frac{C}{14}t$-connected subgraph $H$ and list assignment $L'\subseteq L$ of $H$ such that $H$ is not $L'$-colorable and 
$$|L'|\ge |L| - (6Ct \log( \max\{|S|,4t\}/ 2t)^2 + 98t)- 4\frac{C}{14}t,$$
and $H$ is $\mc{B}$-tangent to $(\mc{A,B})$ and $(V(H)\setminus V(\mc{A,B})) \cap V(\mc{R}) = \emptyset$.
\end{itemize}
\end{lem}
\begin{proof}
We show that 
$$C_{\ref{lem:inseparable2}} = 14 \cdot \lceil \max \{48+2C_{\ref{t:SmallConn}}, 10C_{\ref{t:linked}}, 9C_{\ref{t:woven}}, 120\cdot (1+C_{\ref{l:listprob}})\cdot(1+C_{\ref{t:logbip2}})^2\cdot(1+ C_{\ref{t:SmallConn}})^2 \} \rceil$$ 
suffices. Suppose not. Let $t$ be integer where there exists a counterexample for some $T$, $G$ and $L$ as in the hypotheses of the lemma. Let $T$, $G$ and $L$ be a counterexample for that value $t$ such that $T$ is minimized.  

First suppose that $T=0$. By Theorem~\ref{t:LargeL}, there exists an induced $\frac{C_{\ref{lem:inseparable2}}}{14}t$-connected subgraph $H$ of $G$ and list assignment $L'\subseteq L$ of $H$ such that $H$ is not $L'$-colorable and $|L'|\ge |L|-4\frac{C_{\ref{lem:inseparable2}}}{14}t$. Let $S$ be a set of $t$ vertices in $H$. Then $S$ is a $K_{0,t}$ model $(\mc{A,B})$ such that $H$ is $\mc{B}$-tangent to $(\mc{A,B})$ as desired. 

Hence we may assume that $T>0$. Let $x=\lceil \log t \rceil$ and $y= \lceil \frac{t}{\log t} \rceil$. Note that $xy \le (\log t + 1)(\frac{t}{\log t} + 1) \le 4t$ since $t\ge 3$.

By the minimality of $T$, there exists
\begin{itemize}
\item a \emph{railroad} $\mathcal{R} = (S, \mathcal{T}_1, \mathcal{T}_2)$ of \emph{width} $t$ and \emph{length} $T-1$ where $|S| \le 2Tt + 3 (T-1) t \cdot f_{\ref{t:SmallConn}}(t) \cdot \log^2 t + 72 C_{\ref{t:logbip2}} (\log t) t \frac{T(T-1)}{2}$, and
\item a $K_{s,t}$ model $(\mc{A,B})$ where $s = (T-1) \lceil \frac{t}{\log t}\rceil$ and $V(\mc{A,B})\subseteq V(\mc{R})$, and
\item a $\frac{C_{\ref{lem:inseparable2}}}{14}t$-connected subgraph $H$ and list assignment $L'\subseteq L$ of $H$ such that $H$ is not $L'$-colorable and
$$|L'|\ge |L| - (6C_{\ref{l:listprob}}t \log( \max\{|S|,4t\}/ 2t)^2 + 98t)- 4\frac{C_{\ref{lem:inseparable2}}}{14}t,$$ 
and $H$ is $\mc{B}$-tangent to $(\mc{A,B})$ and $(V(H)\setminus V(\mc{A,B})) \cap V(\mc{R}) = \emptyset$.
\end{itemize}

\noindent As $|T|\le (\log t)^2$ and $f_{\ref{t:SmallConn}}(t) \le C_{\ref{t:SmallConn}}\log \log t$, we have that 

\begin{align*}
|L'| &\ge |L| - (6C_{\ref{l:listprob}}t \log (40T^2 (1+C_{\ref{t:logbip2}})(1+f_{\ref{t:SmallConn}}(t)) \cdot \log^2 t)^2 + 98t) - 4\frac{C_{\ref{lem:inseparable2}}}{14}t \\
&\ge |L| - (180\cdot C_{\ref{l:listprob}}\cdot (1+C_{\ref{t:logbip2}})^2(1+C_{\ref{t:SmallConn}})^2\cdot t(\log\log t)^2) - 4\frac{C_{\ref{lem:inseparable2}}}{14} t\\
&\ge |L| - 5\frac{C_{\ref{lem:inseparable2}}}{14} t (\log\log t)^2 \\
&\ge |L| - C_{\ref{lem:inseparable2}} t (\log\log t)^2,
\end{align*}
where for the second to last inequality we used the fact that $\frac{C_{\ref{lem:inseparable2}}}{14} \ge 180\cdot (1+C_{\ref{l:listprob}})\cdot(1+C_{\ref{t:logbip2}})^2\cdot (1+C_{\ref{t:SmallConn}})^2$.

For each $j\in[t]$, let $w_j$ be the unique vertex in $V(H)\cap B_j$. Let $W=\{w_i:i\in[t]\}$. Let $H_1 = H\setminus W$. Since $|L'|\ge t$, there exists an $L'$-coloring $\phi$ of $G[W]$; let $L_1(v) = L'(v)\setminus \{\phi(u): u\in N_H(v)\cap W\}$ for all $v\in V(H_1)$.

Let $k=\frac{C_{\ref{lem:inseparable2}}}{14}t$. Since $\frac{C_{\ref{lem:inseparable2}}}{2}\ge 14\cdot 6(C_{\ref{l:listprob}}+1)(C_{\ref{t:SmallConn}}+1)$, it follows that 
$$\frac{C_{\ref{lem:inseparable2}}}{2}t(\log \log t)^2 \ge 6(C_{\ref{l:listprob}}+1)t (\log f_{\ref{t:SmallConn}}(t) + \log \log t + 1)^2$$ 
Thus 
\begin{align*}
|L_1|&\ge |L| - \frac{5C_{\ref{lem:inseparable2}}}{14}t (\log \log t)^2 - t \\
&\ge C_{\ref{lem:inseparable2}}\cdot t \cdot (f_{\ref{t:SmallConn}}(t)+(\log \log t)^2) - \frac{C_{\ref{lem:inseparable2}}}{2}t \log \log t\\
&\ge C_{\ref{lem:inseparable2}}\cdot t \cdot f_{\ref{t:SmallConn}}(t) + \frac{C_{\ref{lem:inseparable2}}}{2}t (\log \log t)^2\\
&\ge 2k \cdot f_{\ref{t:SmallConn}}(t) + 6(C_{\ref{l:listprob}}+1)t (\log f_{\ref{t:SmallConn}}(t) + \log \log t + 1)^2.
\end{align*}
Given this and the fact that $H_1$ has no bipartite $K_t$ minor, we have that $t, k, H_1$ and $L_1$ satisfy the hypotheses of Corollary~\ref{c:SmallConn}.

Thus by Corollary~\ref{c:SmallConn}, there exist $x+1$ vertex-disjoint $k$-connected subgraphs $J_1, \ldots J_{x}, D$ in $H_1$ such that $\v(J_i) \le t \cdot f_{\ref{t:SmallConn}}(t) \cdot \log t$ for each $i\in[x]$ and $\v(D) \le t \cdot f_{\ref{t:SmallConn}}(t) \cdot \log t$. For each $i\in[x]$, let $\{w_{t+(i-1)2y+1},\ldots w_{t+2iy}\}$ be a subset of $V(J_i)$ of size $2y$. Let $W_1 = \{w_i: i\in [t+2xy]\}$.

Now $|W_1|\le t+ 2xy \le 9t$. As $H$ is $18t$-connected since $\frac{C_{\ref{lem:inseparable2}}}{14} \ge 18$, it follows from Menger's theorem that there exists a set $\mc{Q}_1$ of $W_1-V(D)$ paths in $H$ with $|\mc{Q}_1|=2|W_1|$ which are vertex-disjoint except in $W_1$ and where each vertex in $W_1$ is the end of exactly two paths in $\mc{Q}_1$. We may assume that $\mc{Q}_1$ is geodesic in $G\setminus (V(\mc{R})\setminus W)$ and hence by Lemma~\ref{l:geodesic2} that $\mc{Q}_1$ is express in $G\setminus (V(\mc{R})\setminus W)$. 

Let $J=V(D)\cup \bigcup_{i=1}^{x} J_i$.  Note then that $|J|\le 3t\cdot f_{\ref{t:SmallConn}}(t) \cdot \log ^2 t$. Since $\mc{Q}_1$ is express and $|\mc{Q}_1|\le 18t$, we have by definition that $G[V(\mc{Q}_1)\setminus (V(J)\cup w)]$ is $36t$-degenerate and every vertex in $G\setminus ((V(\mc{R})\setminus W)\cup V(\mc{Q}_1))$ has at most $54t$ neighbors in $V(\mc{Q}_1)\setminus (V(J)\cup w)$.
 

Let $S'=S\cup J \cup W$ and $\mc{T}_1' = \mc{T}_1 \cup \{P\setminus W: P\in \mc{Q}_1\}$. Let $H_2 = G\setminus (S'\cup V(\mc{T}_1')\cup V(\mc{T}_2))$.

Let $X=S'$ and $Y=V(\mc{T}_1')\cup V(\mc{T}_2)$. Since $\mc{R}$ is a railroad, it follows by definition that every vertex in $G\setminus V(\mc{R})$ has at most $18t$ neighbors in $V(\mc{T}_1)\cup V(\mc{T}_2)$ and hence at most $72t$ neighbors in $Y$. Moreover, it also follows that $G[Y]$ is $90t$-degenerate. Note that 
\begin{align*}
|X| = |S'| &= |S|+|J|+|W| \\
&\le  2Tt + 3 (T-1) t \cdot f_{\ref{t:SmallConn}}(t) \cdot \log^2 t + 3t\cdot f_{\ref{t:SmallConn}}(t) \cdot \log ^2 t + 72 C_{\ref{t:logbip2}} (\log t) t \cdot \frac{T(T-1)}{2} + t \\
&= (2T+1)t + 3 T t \cdot f_{\ref{t:SmallConn}}(t) \cdot \log^2 t + 72 C_{\ref{t:logbip2}} (\log t) t \cdot \frac{T(T-1)}{2}.
\end{align*}
Let $d= 90t$. Note also that 
$$|L| > 6C_{\ref{l:listprob}}t \log \left( \frac{\max\{|X|,4t\}}{2t} \right)^2 + d + 2t.$$

By Lemma~\ref{l:listprob}, there exists a list assignment $L_2\subseteq L$ of $H_2=G-X-Y$ such that $H_2$ is not $L_2$-colorable and
\begin{align*}
|L_2|&\ge |L| - (6Ct \log( \max\{|X|,4t\}/ 2t)^2 + d+2t).\\
&\ge |L| - (6C_{\ref{l:listprob}}t \log (76T^2 (1+C_{\ref{t:logbip2}})(1+f_{\ref{t:SmallConn}}(t)) \cdot \log^2 t)^2 + 98t) \\
&\ge |L| - (180\cdot C_{\ref{l:listprob}}\cdot (1+C_{\ref{t:logbip2}})^2(1+C_{\ref{t:SmallConn}})^2\cdot t(\log\log t)^2) \\
&\ge |L| - \frac{C_{\ref{lem:inseparable2}}}{14} t (\log\log t)^2,
\end{align*}
where for the second to last inequality we used the fact that $\frac{C_{\ref{lem:inseparable2}}}{14} \ge 180\cdot (1+C_{\ref{l:listprob}})\cdot(1+C_{\ref{t:logbip2}})^2\cdot (1+C_{\ref{t:SmallConn}})^2$.


By Theorem~\ref{t:LargeL} as $|L_2|\ge 7k$, there exists an induced $k$-connected subgraph $H_3$ of $H_2$ and a list assignment $L_3\subseteq L_2$ of $H_3$ such that $H_3$ is not $L_3$-colorable and 
$$|L_3|\ge |L_2|-4k.$$

First suppose that $|V(H_3)\cap V(H)|\le k$. Let $H_4 = H_3\setminus V(H)$. Since $|L_3|\ge k$, there exists an $L_3$-coloring $\phi$ of $G[V(H_3)\cap V(H)]$; let $L_4(v) = L_3(v)\setminus \{\phi(u): u \in N_{H_3}(v)\cap V(H)\}$. Note that $H_4$ is not $L_4$-colorable and
\begin{align*}
|L_4|&\ge |L_3|-k \\
&\ge |L_2|-5k \\
&\ge  |L| - 6\frac{C_{\ref{lem:inseparable2}}}{14} t (\log\log t)^2\\
&\ge |L| - C_{\ref{lem:inseparable2}} t (\log\log t)^2.
\end{align*}
Now $H_4$ and $H$ are disjoint and $H$ is not $L'$-colorable while $H_4$ is not $L_4$-colorable and 
$$|L'|,|L_4| \ge |L| - C_{\ref{lem:inseparable2}} t (\log\log t)^2.$$ 
Thus $L$ is $C(t\log \log t)^2$-chromatic-separable in $G$, a contradiction.

So we may assume that $|V(H_3)\cap V(H)|\ge k$.

Since $H$ and $H_3$ are $k$-connected and $|V(H_3)\cap V(H)|\ge k$, it follows that $H_3\cup H$ is $k$-connected. Since $k\ge 11t$ as $\frac{C_{\ref{lem:inseparable2}}}{14} \ge 11$, it follows that $(H\cup H_3)\setminus W_1$ is $2t$-connected. By Menger's theorem, there exist a set $\mc{Q}_2$ of vertex-disjoint $V(H_3)-D$ paths in $(H\cup H_3)\setminus W_1$ such that $|\mc{Q}_2|=2t$. We may assume that $\mc{Q}_2$ is geodesic in $G-X-Y$ and hence by Lemma~\ref{l:geodesic1} that $\mc{Q}_2$ is express in $G-X-Y$.

Let $a_1,\ldots, a_t, b_1,\ldots b_t$ be the vertices in $V(H_3)$ that are the ends of the paths in $\mc{Q}_2$. As $H_3$ is $C_{\ref{t:linked}}t$-connected since $\frac{C_{\ref{lem:inseparable2}}}{14} \ge C_{\ref{t:linked}}$, there exists by Theorem~\ref{t:linked} an $(\{a_i,b_i\})_{i\in [t]}$ linkage in $H_3$. As $H_3$ is connected, it now follows that there exists a partition $C_1,\ldots, C_t$ of $V(H_3)$ such that $H_3[C_i]$ is connected and $a_i,b_i\in C_i$ for each $i\in [t]$.

Let $H'$ be obtained from $H$ by for each $i\in [t]$, contracting each $C_i$ to a vertex denoted $c_i$. Let $C=\bigcup_i c_i$.

Now by Lemma~\ref{lem:MengerVariant} given the existence of $\mc{Q}_1$ and $\mc{Q}_2$, it follows that there exists a set $\mc{P}'_1$ of vertex-disjoint paths $(W_1\cup C) - V(D)$ paths in $H'$ such that $|\mc{P}'_1|=|Z|+t$. 

For each, $i \in [t+2xy]$, let $P_i$ be the path in $\mc{P}'_1$ containing $w_i$ and let $u_i$ denote the end of $P_i$ in $D$. For each $i\in [t]$, let $P'_{t+2xy+i}$ be the path in $\mc{P}'_1$ containing $c_i$ and let $u_{t+2xy+i}$ denote the end of $P'_{t+2xy+i}$ in $D$.

For each $i\in [t]$, $P'_{t+2xy+i}$ corresponds to a $V(H_3)-V(D)$ path $P_{t+2xy+i}$ in $H$ from a vertex $w_{t+2xy+i}$ in $C_i$ to $u_{t+2xy+i}$. Now $\mc{P}=\{P_i: i\in [2t+2xy]\}$ is a $(\{w_i,u_i\})_{i\in [2t+2xy]}$-linkage such that $V(\mc{P})\cap V(H_3) = \{w_{t+2xy+i}: i\in [t]\}$ and $V(\mc{P})\cap V(D) = \{u_i: i\in [2t+2xy]\}$. 


For $i\in [x]$, since $J_i$ is $9C_{\ref{t:woven}}t$-connected as $\frac{C_{\ref{lem:inseparable2}}}{14} \ge 9C_{\ref{t:woven}}$, we have by Lemma~\ref{t:woven} that $J_i$ is $(2y,2t+2xy)$-woven. 

\begin{claim}\label{Linkages1}
There exists a $(\{w_i,u_i\})_{i\in[2t+2xy]}$-linkage $\mc{P}'$ with $V(\mc{P}')\subseteq V(\mc{P})$ and $K_{2y}$ models $\mc{M}_i$ in $J_i$ rooted at $X_i$ for $i\in[x]$ such that $V(\mc{P}')\cap V(\mc{M}_i) = X_i$ for each $i\in [x]$.
\end{claim}
\begin{proof}
Since each $J_i$ is $(2y,2t+xy)$-woven, the claim follows by iteratively applying Lemma~\ref{l:woven} to each $J_i$.
\end{proof}

By Claim~\ref{Linkages1}, there exists $\mc{P}'$ and $\mc{M}_i$ for each $i\in [x]$ as in the statement of the claim. Now let us define a few parameters. For each $i\in [y]$ and $j\in [x]$, define $$a_{j,i}:= t+(j-1)2xy + y + i,$$ 
and let $M_{a_{j,i}}$ denote the subgraph in $\mc{M}_j$ containing $w_{a_{j,i}}$. For each $j\in [t]$, define 
$$b_j := t+j + y \left(\left\lceil \frac{j}{y} \right\rceil-1\right),$$
$$b'_j := t+2xy+j,$$
and let $M_{b_j}$ denote the subgraph in $\mc{M}_{\left\lceil \frac{j}{y} \right\rceil}$ containing $w_{b_j}$. For each $j\in[t]$, let $$S_j := \{u_j, u_{b_j},u_{b'_j}\}.$$ For each $i\in [y]$, let 
$$S_{t+i} := \{ u_{a_{j,i}} : j\in[x]\}.$$

Note that $D$ is $(2t+2xy)C_{\ref{t:linked}}$-connected as $2t+2xy\le 10t$ and $10C_{\ref{t:linked}} \le \frac{C_{\ref{lem:inseparable2}}}{14}$. Hence by Theorem~\ref{t:linked}, it follows that there exist vertex-disjoint connected subgraphs $D_1,\ldots D_{t+y}$ of $D$ where $S_j\subseteq D_j$ for each $j\in [t+y]$.

Now we construct a $K_{s+y, t}$ model $(\mc{A',B'})$ in $G$ such that $H_3$ is $\mc{B'}$-tangent to $(\mc{A',B'})$ as follows:
\begin{itemize}
\item For $i\in [s]$, let $A'_i := A_i$.
\item For $i\in[y]$, let $A'_{s+i} := D_{t+i} \cup \bigcup_{j\in[x]} \left( P_{a_{j,i}} \cup \mc{M}_{a_{j,i}} \right)$.
\item For $j\in [t]$, let $B'_j := B_j \cup P_j \cup D_j \cup P_{b_j} \cup \mc{M}_{b_j} \cup P_{b'_j}$.
\end{itemize}
Note that $s+y = (T-1)y+y = T \left \lceil \frac{t}{\log t} \right \rceil$.

Since $\mc{Q}_2$ is express in $G-X-Y$ and $|\mc{Q}_2|=2t$, we have by definition that every vertex in $G\setminus (X\cup Y \cup (V(\mc{Q}_2)\setminus V(H_3)))$ has at most $3$ neighbors in any path in $\mc{Q}_2$. Let $\mc{T}_2' = \mc{T}_2 \cup \{P\setminus V(H_3): P\in \mc{Q}_2\}$. Thus every vertex in $G\setminus (X\cup Y \cup (V(\mc{Q}_2)\setminus V(H_3))$ has at most $3$ neighbors in any path in $\mc{T}_1'\cup \mc{T}_2'$. 

Moreover, $|\mc{T}_1'| = |\mc{T}_1| + |\mc{Q}_1| \le 18t(T-1) + 18t \le 18tT$. Similarly, $|\mc{T}_2'| = |\mc{T}_2|+|\mc{Q}_2| \le 18t(T-1) + 2t \le 18tT$. In addition, the paths in $\mc{T}_1'$ are pairwise vertex-disjoint and similarly the paths in $\mc{T}_2'$ are pairwise vertex-disjoint.

For each $j\in \{1,2\}$, let $\phi_j$ be a $2$-coloring of the paths in $\mc{T}_j'$. For each $i,j\in \{1,2\}$, let $S_{i,j}$ be the set of vertices $v$ in $G\setminus (X\cup Y \cup V(\mc{Q}_2))$ such that $|\{P\in \mc{T}_j': \exists u\in V(P)\cap N_G(v) \text{ with } \phi(u)=i\}| \ge 2t$. For each $j\in \{1,2\}$, let $G_j$ be the graph obtained from $G$ by contracting each path $P$ in $\mc{T}_j'$ to a vertex $z_P$ and let $Z_j = \{ z_P: P\in \mc{T}_j'\}$. Note that for each $j\in \{1,2\}$, $G_j$ is a minor of $G$. It follows from Theorem~\ref{t:logbip2} that for each $i,j\in\{1,2\}$, if $|S_{i,j}| \ge C_{\ref{t:logbip2}} (\log t) |\mc{T}_j'|$, then $G$ has a bipartite $K_t$ minor, a contradiction. So we may assume that $|S_{i,j}| \le  C_{\ref{t:logbip2}} (\log t) |\mc{T}_j'|\le 18 C_{\ref{t:logbip2}} (\log t) t T$ for each $i,j\in\{1,2\}$.

Let $S_0 = \bigcup_{i,j\in \{1,2\}} S_{i,j}$. Note that $|S_0|\le 72 C_{\ref{t:logbip2}} (\log t) t T$. Let $S'' = S'\cup S_0\cup \{u_{t+2xy+i}: i\in[t]\}$. Thus

\begin{align*}
|S''| \le |S'|+|S_0| &\le (T+1)t + 3 T t \cdot f_{\ref{t:SmallConn}}(t) \cdot \log^2 t + 72 C_{\ref{t:logbip2}} (\log t) t \cdot \frac{T(T-1)}{2} + 72 C_{\ref{t:logbip2}} (\log t) t T + t\\
&\le 2(T+1)t + 3 T t \cdot f_{\ref{t:SmallConn}}(t) \cdot \log^2 t + 72 C_{\ref{t:logbip2}} (\log t) t \cdot \frac{T(T+1)}{2}
\end{align*}

Now let $\mc{R'}= (S'', \mc{T}_1', \mc{T}_2')$. Note that $\mc{R'}$ is a railroad of width $t$ and length $T$. Moreover $G[V(\mc{R'})]$ contains the $K_{s+y,t}$ model $(\mc{A',B'})$ and $H_3$ is a $\frac{C_{\ref{lem:inseparable2}}}{14}t$-connected subgraph and $L_3\subseteq L$ is list assignment of $H_3$ such that $H_3$ is not $L_3$-colorable and 
$$|L_3|\ge |L| - (6C_{\ref{l:listprob}}t \log( \max\{|S''|,4t\}/ 2t)^2 + 98t)- 4\frac{C_{\ref{lem:inseparable2}}}{14}t,$$
and $H_3$ is $\mc{B'}$-tangent to $(\mc{A',B'})$ and $(V(H_3)\setminus V(\mc{A',B'})) \cap V(\mc{R'}) = \emptyset$, a contradiction.

\end{proof}

We are now ready to prove Lemma~\ref{lem:inseparable} as a direct corollary of Lemma~\ref{lem:inseparable2}. We restate Lemma~\ref{lem:inseparable} for convenience.

\Inseparable*
\begin{proof}[Proof of Lemma~\ref{lem:inseparable}]
Suppose not. Apply Lemma~\ref{lem:inseparable2} with $t_0:=4\max\{C_{\ref{t:logbip3}},1\}t$ and $T=\lceil (\log t_0)^2 \rceil$. Then $G$ contains a $K_{s,t_0}$ model $(\mc{A,B})$ where $s = T \lceil \frac{t_0}{\log t_0}\rceil \ge \lceil 4C_{\ref{t:logbip3}}t\log t \rceil$. Hence $G$ contains a $K_{4t,\lceil 4C_{\ref{t:logbip3}}t\log t \rceil}$ minor. Thus by Theorem~\ref{t:logbip3}, $G$ contains a bipartite $K_t$ minor, a contradiction.
\end{proof}
 
\section{Finding a Minor via Chromatic-Separable List Assignments}\label{s:separable}

We prove a stronger inductive form of Lemma~\ref{lem:separable} as follows.

\begin{lem}\label{l:rooted2} There exists an integer $C=C_{\ref{l:rooted2}} >0$ such that the following holds:  Let $t\ge 3$ be an integer and let $T\in [0, \lceil 4 \log \log t \rceil]$ be an integer. Suppose that $k$ and $m$ are integers with $k\ge C m \log \log t$ and $m\ge Ct$. If $G$ is a $Ct$-connected graph and $L$ is a $(k - 8T\cdot m - 6Ct)$-list-assignment of $G$ such that 
\begin{itemize}
\item $G$ is not $L$-colorable, and 
\item for every subgraph $H$ of $G$ and every list assignment $L'\subseteq L$ of $H$ with $|L'| \ge \frac{7}{8}k$ and $H$ is not $L'$-colorable, we have that $L'$ is $m$-chromatic-separable in $H$,
\end{itemize}
then $G$ is $(a,b)$-woven for all integers $a,b$ where

\begin{itemize}
\item $a \le (2/3)^T \cdot t$, and 
\item $b \le 2t \cdot \sum_{j=0}^T \left(\frac{2}{3}\right)^j$.
\end{itemize}

Furthermore if $L(v) \subseteq [k]$ for all $v\in V(G)$, then $G$ is $(a,b)$-parity-woven.
\end{lem}
\begin{proof}
Note that it suffices to prove the lemma when $t$ is a power of $3$ at the cost of increasing $C_{\ref{l:rooted2}}$ by a factor of $3$. We now assume this for the rest of the proof. Let 
$$C_{\ref{l:rooted2}} = \lceil \max\{6\cdot C_{\ref{t:linked}}+1,6\cdot C_{\ref{t:woven}},6\cdot C_{\ref{t:paritywoven}},416\} \rceil.$$

Suppose not. Let $G$ be a counterexample with $T$ maximized. Let $a = (2/3)^T \cdot t$, and $b = 2t \cdot \sum_{j=0}^T \left(\frac{2}{3}\right)^j$. Since $t$ is a power of $3$, $a$ and $b$ are integers.  Note that $a\le t$ and $b\le 6t$.

Let $R=\{r_1,\ldots, r_a\}$ be a set of vertices in $G$, and let $S=\{s_1,\ldots, s_b\},$ and $T=\{t_1,\ldots,t_b\}$ be multisets of vertices in $G$. Furthermore if $L(v) \subseteq [k]$ for all $v\in V(G)$, let $J\subseteq [a]$ and $I \subseteq [b] \setminus \{i: s_i=t_i\}$,

First suppose $T = \lceil 4 \log \log t \rceil$. Then $a \le t/ \log t$ and hence $a \sqrt{\log a} \le t$. Yet $b\le 6t$. Since $C_{\ref{l:rooted2}} \ge 6\cdot C_{\ref{t:woven}}$, it follows from Theorem~\ref{t:woven} that $G$ is $(a,b)$-woven. Furthermore, if $L(v) \subseteq [k]$ for all $v\in V(G)$, then since $C_{\ref{l:rooted2}} \ge 6\cdot C_{\ref{t:paritywoven}}$, it follows from Theorem~\ref{t:woven} that $G$ is $(a,b)$-parity-woven, a contradiction. So we may assume that $T < 4 \log \log t$.

For each $i\in[a]$, let $s_{b+i}=s_{b+a+i}=r_i$. Let $S' = S\cup \{s_i : i\in \{b+1, \ldots, b+2a\} \}$.

Let $G'=G\setminus (R\cup S'\cup T)$. Since $|L|\ge a+2b$, there exists an $L$-coloring $\phi$ of $G[R\cup S'\cup T]$. For each $v\in V(G')$, let $L'(v) = L(v)\setminus \{\phi(u):u\in N_G(v)\}$.

Note that $L'\subseteq L$, $G'$ is not $L'$-colorable and
$$|L'|\ge |L|-(a+2b) \ge |L|-13t \ge k-3T\cdot m - 6Ct - 13t \ge \frac{7}{8}k,$$ 
where the last inequality follows since $T < 4 \log \log t$ and $C_{\ref{l:rooted2}} \ge 416 > 8(3\cdot 4 + 13 + 5)= 240$. Hence by assumption, $L'$ is $m$-chromatic-separable in $G'$. Thus by definition of $m$-chromatic-separable, there exist two vertex-disjoint subgraphs $H_1,H_2$ of $G'$ such that for each $i\in \{1,2\}$, there exists a list assignment $L_i\subseteq L$ of $H_i$ such that $H_i$ is not $L_i$-colorable and
$$|L_i| \ge |L|-m.$$

Similarly since $C_{\ref{l:rooted2}} \ge 416 > 248$, it also follows by assumption that $L_1$ is $m$-chromatic-separable in $H_1$. Hence by definition there exist two vertex-disjoint subgraphs $H_{1,1}$, $H_{1,2}$ of $H_1$ such that for each $i\in \{1,2\}$, there exists a list assignment $L_{1,i}\subseteq L_1\subseteq L$ of $H_{1,i}$ such that $H_{1,i}$ is not $L_{1,i}$-colorable and
$$|L_{1,i}| \ge |L_1| - m \ge |L|-2m.$$

Let $J_1 = H_{1,1}, J_2= H_{1,2}$ and $J_3=H_{2}$ and let $L_1' = L_{1,1}$, $L_2' = L_{1,2}$ and $L_3'=L_2$. Note that for each $i\in [3]$,
$$|L_i'|\ge |L|-2m \ge k - 8T\cdot m - 6Ct \ge k - 8(T+1)\cdot m \ge 7 C_{\ref{l:rooted2}}t,$$
where the middle inequality follows since $m\ge C_{\ref{l:rooted2}}t$ and the last inequality follows since $C_{\ref{l:rooted2}} \ge 416 > 8\cdot(5+1)+7 = 55$. Hence for each $i\in [3]$, by Theorem~\ref{t:LargeChi} as $|L_i|\ge 7 C_{\ref{l:rooted2}}t$, there exists a $C_{\ref{l:rooted2}}t$-connected subgraph $J'_i$ of $J_i$ and list assignment and list assignment $L_i''\subseteq L_i'$ such that $J_i'$ is not $L_i''$-colorable and
$$|L_i''|\ge |L_i'|-4k.$$
Note that
$$|L_i''|\ge |L_i'| - 4C_{\ref{l:rooted2}} t \ge k - 8(T+1)m - 4C_{\ref{l:rooted2}} t.$$ 

Since $T< 4\log \log t$ and $C_{\ref{l:rooted2}}\ge 8\cdot (8\cdot(5+1)+4)=416$, we have that  $|L_i''|\ge \frac{7}{8}k$ and hence $\v(J'_i)\ge t$. Let $a' = 2a/3$. Note that $a' \le (2/3)^{T+1} t$.  Similarly let $b'=b + 2a'$. Note then that
$$b' = b + 2a' = 2t \cdot \sum_{j=0}^T \left(\frac{2}{3}\right)^j + 2\left(\frac{2}{3}\right)^{T+1} t = 2t \cdot \sum_{j=0}^{T+1} \left(\frac{2}{3}\right)^j.$$ 
Moreover for each $i\in [3]$, every subgraph $H$ of $J'_i$ and list assignment $L_H$ with $|L_H|\ge \frac{7}{8}k$ where $H$ is not $L_H$-colorable is $m$-chromatic-separable since every such subgraph of $G$ was by assumption. Thus by the maximality of $T$, $J'_i$ is $(a',b')$-woven. Furthermore if $L(v) \subseteq [k]$ for all $v\in V(G)$, then by the maximality of $T$, $J'_i$ is $(a',b')$-parity-woven.

For each $i\in [3]$, since $\v(J'_i)\ge t$, there exists a subset $T_i = \{t_{b+(i-1)a'+1}, \ldots, t_{b+ia'}\}$ of $V(J'_i)$. Let $T'=\{t_i : i\in [b+2a]\}$. Note that $|S'|=|T'|=b+2a$. Since $G\setminus R$ is $(C_{\ref{l:rooted2}}-1)t$-connected and $C_{\ref{l:rooted2}} \ge 6 C_{\ref{t:linked}}+1$, we have by Theorem~\ref{t:linked} that there exists an $(\{s_i,t_i\})_{i\in [b+2a]}$-linkage $\mc{P'}$. Furthermore if $L(v) \subseteq [k]$ for all $v\in V(G)$, then since $C_{\ref{l:rooted2}} \ge 416\ge 401$ and $|L_i''|\ge \frac{7}{8}k$, we have by Theorem~\ref{t:paritylinked2} that there exists an $((\{s_i,t_i\})_{i\in [b+2a]},I)$-linkage $\mc{P'}$.

Let $Y_i = \{t_{(i-1)a'+j}: j\in[a']\}$.


\begin{claim}\label{Linkages2}
There exists an $(\{s_i,t_i\})_{i\in [b+2a]}$-linkage $\mc{P''}$ in $G$ and $K_{a'}$ models $\mc{M}_i$ in $J'_i$ rooted at $Y_i$ for each $i\in [3]$ such that $V(\mc{P''})\cap V(\mc{M}_i) = Y_i$ for each $i\in [3]$. Furthermore if $L(v) \subseteq [k]$ for all $v\in V(G)$, then there exists a $K_{a}$ bipartite expansion $\mc{M}_i$ in $J_i'$ rooted at $Y_i$ with color classes $A_i$ and $B_i$ such that $Y_i\subseteq A_i$ and an $((\{s_i,t_i\})_{i \in [b+2a]}, I)$-linkage $\mc{P''}$ in $G$ such that $V(\mc{P''})\cap V(\mc{M}_i) = Y_i$ for each $i\in [3]$.
\end{claim}
\begin{proof}
Since each $J'_i$ is $(a',b')$-woven, the claim follows by iteratively applying Lemma~\ref{l:woven} to each $J'_i$. Furthermore if $L(v) \subseteq [k]$ for all $v\in V(G)$, then since each $J'_i$ is then $(a',b')$-parity-woven, the claim follows instead by iteratively applying Lemma~\ref{l:paritywoven} to each $J'_i$.
\end{proof}

Now back to the main proof. By Claim~\ref{Linkages2}, there exist $\mc{P''}$ and $K_{a'}$ models $\mc{M}_i$ as in the claim. For each $i\in[b+2a]$, let $P''_i$ be the path in $\mc{P''}$ containing $\{s_i,t_i\}$. Let 
$$\mc{P} = \{P''_i: i\in [b]\}.$$ 
Note that $\mc{P}$ is an $(\{s_i,t_i\})_{i \in [b]}$-linkage in $G$. For each $i\in [3]$ and $j\in [a']$, let $M'_{(i-1)a'+j}$ denote the subgraph in $\mc{M}_i$ containing $t_{(i-1)a'+j}$. For each $i\in [a]$, let 
$$M_i = M'_{b+i} \cup P''_{b+i} \cup \{r_is_{b+i}, r_is_{b+a+i}\} \cup P''_{b+a+i} \cup M'_{b+a+i}.$$ 
Now $\mc{M} = \{M_i : i\in [a]\}$ is a $K_a$ model in $G$ rooted at $R$. Moreover $V(\mc{M})\cap V(\mc{P}) = R\cap (S\cup T)$. Thus $G$ is $(a,b)$-woven. 

Furthermore if $L(v) \subseteq [k]$ for all $v\in V(G)$, then $\mc{P}$ is an $((\{s_i,t_i\})_{i \in [b]},I)$-linkage in $G$ and $\mc{M}$ is a $K_{a}$ bipartite expansion $\mc{M}$ in $G$ rooted at $R$ with color classes $A$ and $B$ such that $A\cap R = J$. Thus $G$ is $(a,b)$-parity-woven. In all cases, we obtain a contradiction. 
\end{proof}

We are now ready to prove Lemma~\ref{lem:separable}, which we restate for convenience.

\Separable*
\begin{proof}[Proof of Lemma~\ref{lem:separable}]
Let $C_{\ref{lem:separable}} = 7C_{\ref{l:rooted2}}$. Let $k = |L|$. By Theorem~\ref{t:LargeChi} as $|L|\ge C_{\ref{lem:separable}}t \ge 7C_{\ref{l:rooted2}} t$, there exists a $C_{\ref{l:rooted2}}t$-connected subgraph $H$ of $G$ and list assignment $L'\subseteq L$ of $H$ such that $H$ is not $L'$-colorable and
$$|L'|\ge |L|-4C_{\ref{l:rooted2}}t.$$
Note that $|L|\ge C_{\ref{l:rooted2}}  m \log \log t$ and $m\ge C_{\ref{l:rooted2}} t$. Also note that this implies that $\v(G)\ge t$. It now follows from Lemma~\ref{l:rooted2} with $T=0$ that $G$ is $(t,0)$-woven and hence contains a $K_t$ minor as desired.

Furthermore if $L(v) = [k]$ for all $v\in V(G)$, then $G$ is $(t,0)$-parity-woven and hence contains a bipartite $K_t$ minor as desired.

\end{proof}


\bibliographystyle{alpha}
\bibliography{lpostle}

\end{document}